\documentclass[11pt]{amsart}
\usepackage{amsmath}
\usepackage{amsfonts}
\usepackage{amssymb}
\usepackage{enumerate}
\usepackage{hyperref}

\newtheorem{thm}{Theorem}[section]

\newtheorem{cor}[thm]{Corollary}

\newtheorem{remark}[thm]{Remark}
\newtheorem{lemma}[thm]{Lemma}
\newtheorem{prop}[thm]{Proposition}

\newtheorem{defn}[thm]{Definition}

\newcommand{\bb}[1]{\mathbb{#1}}
\newcommand{\cl}[1]{\mathcal{#1}}

\allowdisplaybreaks

\title[Super Operator System Structures in Quantum Entanglement]{The super operator system structures and their applications in Quantum Entanglement Theory}

\author[B.~Xhabli]{Blerina ~Xhabli}
\address{Department of Mathematics, University of Houston,
Houston, Texas 77204-3476, U.S.A.}
\email{blerinax@math.uh.edu}

\begin{document}

\begin{abstract}
An operator system $\cl S$ with unit $e$, can be viewed as an Archimedean order unit space $(\cl S,\cl S^+,e)$. Using this Archimedean order unit space, for a fixed $k\in \bb N$ we construct a super k-minimal operator system OMIN$_k(\cl S)$ and a super k-maximal operator system OMAX$_k(\cl S)$, which are the general versions of the minimal operator system OMIN$(\cl S)$ and the maximal operator system OMAX$(\cl S)$ introduced recently, such that for $k=1$ we obtain the equality, respectively. We develop some of the key properties of these super operator systems and make some progress on characterizing when an operator system $\cl S$ is completely boundedly isomorphic to either OMIN$_k(\cl S)$ or to OMAX$_k(\cl S)$. Then we apply these concepts to the study of k-partially entanglement breaking maps. We prove that for matrix algebras a linear map is completely positive from OMIN$_k(M_n)$ to OMAX$_k(M_m)$ for some fixed $k\le \min(n,m)$ if and only if it is a k-partially entanglement breaking map. \medskip


\end{abstract}

\maketitle

\section{Introduction}

Operator system theory was initiated with Arveson's version of the Hahn-Banach theorem for completely positive operator-valued mappings~\cite{WA69}. This theory provides an abstract description of the order structure of self-adjoint unital subspaces of C$^*$-algebras. In the past twenty years, beginning with Ruan's abstract characterization of operator spaces \cite{Ru}, there has been a great deal of research activity focused on operator spaces and completely bounded maps. In contrast, there has been relatively little development of the abstract theory of operator systems. However, many deep results about operator
spaces are obtained by regarding them as corners of operator systems. So, potentially, parallel developments in the theory of operator systems could lead to new insights in the theory of operator spaces. 

Moreover, recent investigations in operator space and operator system theory \cite{VIP,PISIER} are being combined with those in quantum entanglement theory \cite{HHHH, NC} in order to obtain new results and new elementary proofs in both areas. From this point of view, the results shown in this paper serve as a bridge between operator system theory and quantum entanglement theory.  These unpublished results~\cite{BX} have been used quite extensively to prove how mapping cones coincide with operator systems \cite{JS}, and also to show the relationship between the operator systems and the separability problem in quantum information theory \cite{JKPP}. We give further details below before proceeding.

In~\cite{PTT}, two operator systems were constructed over a given Archimedean order unit space $\cl S$, denoted as OMIN$(\cl S)$ and OMAX$(\cl S)$, as the analogues of MIN and MAX functors from the category of normed spaces into the category of operator spaces, and their properties were developed accordingly. The properties that characterize these two new formulated operator systems led the authors to prove that the entanglement breaking maps between matrix algebras, studied in \cite{Hol, HSR, Ar}, coincide with the linear maps that are completely positive when the matrix algebra of the domain is equipped with their minimal operator system structure and the target matrix algebra is equipped with their maximal operator system structure.

In this paper, we consider a generalization of such parallel developments for operator systems. 
Every operator system $\cl S$ with a unit is an ordered *-vector space $\cl S$ with an Archimedean order unit at the first level and conversely, given any Archimedean order unit space, there are possibly many different operator systems that all have the given Archimedean order unit space as their first level. For a fixed $k\in \bb N$ and a given operator system $\cl S$, we construct a super k-minimal operator system OMIN$_k(\cl S)$, and a super k-maximal operator system OMAX$_k(\cl S)$, such that whenever $k=1$ we obtain OMIN$(\cl S)$ and OMAX$(\cl S)$ respectively. We investigate their properties in Sections~\ref{oss}, ~\ref{kmax}. Furthermore, we provide necessary and sufficient conditions for an operator system $\cl S$ to be completely boundedly isomorphic to OMIN$_k(\cl S)$ or OMAX$_k(\cl S)$ in these two sections. 

In Section~\ref{mat} we discuss the dual matrix ordered space to a given matrix ordered space and identify the dual spaces of the super operator systems OMIN$_k(\cl S)$ and OMAX$_k(\cl S)$. In Section~\ref{kpeb} we apply our results to the study of partially entanglement breaking maps between matrix algebras encountered in Quantum Information Theory~\cite{CK, NJ}. We characterize the k-partially entanglement breaking maps from $M_n$ to $M_m$ as the maps that are completely positive from OMIN$_k(M_n)$ to OMAX$_k(M_m)$, where $k\le \min(n,m)$. The next section is devoted to preliminary notions and results.\\

\section{Preliminaries}\label{prel}

Let V be a complex vector space. An {\bf involution} on V is a conjugate linear map $*: V \to V$ given by $v\mapsto v^*$, such that $v^{**}=v$ and $(\lambda v+ w)^*=\bar\lambda v^*+w^*$ for all $\lambda\in \bb C$ and $v,w\in V$. The complex vector space V together with the involution map  is called a {\bf $*$-vector space}. If V is a $*$-vector space, then we let $V_{sa}=\{v\in V | v=v^*\}$ be the real vector space of self-adjoint elements of V.\\

A {\bf cone} $W \subseteq V$ is a nonempty subset of a real vector space V, such that $W+W\subseteq W$ and $\bb R^+W \subseteq W$ where $\bb R^+=[0,\infty)$. Moreover, $W$ is called a {\bf proper cone} if $ W \cap (-W) =\{0\}$. An {\bf ordered $*$-vector space $(V, V^+)$} is a pair consisting of a $*$-vector space V and a proper cone $V^+\subseteq V_{sa}$. The elements of $V^+$ are called positive and there is a partial order $\ge$(respectively, $\le$) on $V_{sa}$ defined by $v\ge w$ (respectively, $w\le v$) if and only if $v-w \in V^+$ for $v,w\in V_{sa}$. \\

An element $e\in V_{sa}$ is called an {\bf order unit} for V if for all $v\in V_{sa}$, there exists a real number $t >0$ such that $te\ge v$. This order unit $e$ is called {\bf Archimedean order unit} if whenever $v\in V$ and $te+v \in V^+$ for all real $t>0$, we have that $v\in V^+$. In this case, we call the triple $(V,V^+,e)$ an {\bf Archimedean ordered unital $*$-vector space} or an {\bf AOU space} for short.\\

Let $(V,V^+)$,  $(W,W^+)$ be two ordered *-vector spaces with order units $e, e'$ respectively. A linear map $\phi: V\to W$ is called {\bf positive} if  $\phi(V^+)\subseteq W^+$, and {\bf unital} if it is positive and $\phi(e)=e'$. Moreover, $\phi$ is an {\bf order isomorphism} if $\phi$ is bijective, and both $\phi, \phi^{-1}$ are positive. Note that, if $\phi:V\to W$ is positive, then $\phi(v^*)=\phi(v)^*$ for all $v\in V$.  \\

Let V be a $*$-vector space and let $M_{n,m}(V)$ denote the set of all $n\times m$ matrices with entries in V . The natural addition and scalar multiplication turn $M_{n,m}(V)$ into a complex vector space. We often write $M_{n,m}= M_{n,m}(\bb C)$, and let $\{E_{i,j}\}^{n,m}_{i,j=1}$ denote its canonical matrix unit system. For a given matrix $A\in M_{n,m}$, we write $\bar A, A^t\text{ and } A^*$ for the complex conjugate, transpose and complex adjoint of $A$, respectively. If $n=m$, we write $M_{n,n}=M_n$ and $I_n$ for the identity matrix. The matrix units determine the linear identifications $M_{n,m}(V)\cong M_{n,m}\otimes V\cong V\otimes M_{n,m},$ where  $$v=(v_{ij})\mapsto\sum_{i,j=1}^{n,m}E_{i,j}\otimes v_{ij}\text{ and  }v=(v_{ij})\mapsto\sum_{i,j=1}^{n,m}v_{i,j}\otimes E_{ij},\text{ respectively.}$$  More often than not, we will use the first linear identification with the matrix coefficients on the right. There are two basic natural operations which link the finite matrix linear spaces $M_{n,m}(V)$: the direct sum and the matrix product. Given $v\in M_{n,m}(V)$ and $w\in M_{p,q}(V)$, then we define the direct sum $v\oplus w \in M_{n+p,m+q}(V)$ by $$v\oplus w=\left[\begin{matrix}v & 0\\ 0 & w\end{matrix}\right]\in M_{n+p,m+q}(V).$$ On the other hand, given $A=(a_{ki})\in M_{p,n},\,B=(b_{jl})\in M_{m,q}$ and $v=(v_{ij})\in M_{n,m}(V)$, we define the matrix product $AvB\in M_{p,q}(V)$ by  $$AvB=\left[\sum_{i,j=1}^{n,m} a_{ki}v_{ij}b_{jl}\right]^{p,q}_{k,l=1}\in M_{p,q}(V).$$ Note that, if $V=M_r$ and we use the identification $M_{n,m}(M_r)\cong M_{n,m}\otimes M_r$, then we have for any $X\in M_{p,n},\, a\in M_{n,m}(M_r)$ and $Y\in M_{m,q}$ $$XaY=(X\otimes I_r)a(Y\otimes I_r)\in M_{p,q}(M_r).$$
 Let $V,\,W$ be two $*$-vector spaces. Given a linear map $\phi:V\to W$ and $n,m\in \bb N$, we have a corresponding map $\phi^{(n,m)}:M_{n,m}(V)\to M_{n,m}(W)$ defined by $\phi^{(n,m)}(v)=(\phi(v_{ij}))$. We let $\phi^{(n)}=\phi^{(n,n)}:M_n(V)\to M_n(W).$\\ If we are given $v,\,w,\, A$ and $ B$ as above, then one can easily verify that $$\phi^{(n+p,m+q)}(v\oplus w)=\phi^{(n,m)}(v)\oplus \phi^{(p,q)}(w)$$ and $$\phi^{(p,q)}(AvB)=A\phi^{(n,m)}(v)B.$$
Moreover, if  $\phi:V\to W$ is a linear map and $W=M_k$, then we have for any $X\in M_{p,n},$ $ a\in M_{n,m}(V)$ and $Y\in M_{m,q}$ $$\phi^{(p,q)}(XaY)=X\phi^{(n,m)}(a)Y=(X\otimes I_k)\phi^{(n,m)}(a)(Y\otimes I_k).$$ 
\indent Let V be a $*$-vector space. We define a $*$-operation on $M_n(V)$ by letting $[v_{ij}]^*=[v_{ji}^*]$. With respect to this operation, $M_n(V)$ is a $*$-vector space. We let $M_n(V)_{sa}$ be the set of all self-adjoint elements of $M_n(V)$. Let $ \{C_n\}_{n=1}^\infty$ be a family of proper cones $C_n \subset M_n(V)_{sa}$ for all $n\in \bb N$, such that they are {\bf compatible}, i.e $X^*C_n X \subseteq C_m$ for all $X\in M_{n,m}, \, m\in \bb N$. We call each such $C_n$ a {\bf matrix cone}, the family of these matrix cones a {\bf matrix ordering on V}, and the pair $(V, \{C_n\}_{n=1}^\infty)$ a {\bf matrix ordered $*$-vector space}.\\

Let $(V, \{C_n\}_{n=1}^\infty)$ and $(W, \{C_n'\}_{n=1}^\infty)$ be matrix ordered $*$-vector spaces. Then a linear map $\phi:V\to W$ is called {\bf completely positive} if $\phi^{(n)}(C_n)\subseteq C_n'$ for all $n\in \bb N$. Moreover, $\phi$ is called a {\bf complete order isomorphism} if $\phi$ is invertible and both $\phi,\, \phi^{-1}$ are completely positive.\\

Let $(V, \{C_n\}_{n=1}^\infty)$ be a matrix ordered $*$-vector space. Let $e\in V_{sa}$ be the distinguished order unit for V. Consider the corresponding diagonal matrix $e_n = e \otimes I_n \in M_n(V)_{sa}$ for all $n\in \bb N$, where $I_n$ is the unit of $M_n$. We say that $e$ is a {\bf matrix order unit} for V if $e_n$ is an order unit for the ordered $*$-vector space $(M_n(V),C_n)$ for each $n$. We say $e$ is an {\bf Archimedean matrix order unit} if $e_n$ is an Archimedean order unit for the ordered $*$-vector space $(M_n(V), C_n)$ for each $n$. Finally, we say that the triple $(V, \{C_n\}_{n=1}^\infty, e)$ is an {\bf (abstract) operator system}, if $V$ is a $*$-vector space, $ \{C_n\}_{n=1}^\infty$ is a matrix ordering on V, and $e$ is an Archimedean matrix order unit.\\

The matrix ordering $\{C_n\}_{n=1}^\infty$ such that $(V,\{C_n\}_{n=1}^\infty,e)$ is an operator system with $C_1 = V^+$ is called an {\bf operator system structure}. Given an operator system $(\cl S, \{P_n\}_{n=1}^\infty, e)$ and a unital positive map $\varphi: V\to \cl S$ such that $V^+ = \varphi^{-1}(P_1)$, one obtains an operator system structure on V by setting $C_n =\varphi_n^{-1}(P_n)$. We shall call this {\bf the operator system structure induced by $\varphi$}.
Conversely, given an operator system structure on $V$, by letting $\cl S = V$ and letting $\varphi$ be the identity map, then we see that the given operator system structure is the one induced by $\varphi$.\\

  If $\cl P = \{P_n\}_{n=1}^\infty$ and $\cl Q = \{Q_n\}_{n=1}^\infty$ are two matrix orderings on V , we say that $\cl P$ is {\bf stronger} than $\cl Q$ (respectively, $\cl Q$ is {\bf weaker} than $\cl P$) if $P_n\subseteq Q_n$ for all $n\in \bb N$. Note that $\cl P$ is stronger than $\cl Q$ if and only if for every n, and every $A,B\in M_n(V )_{sa}$, the inequality $A \le_{\cl P} B$ implies that $A\le_{\cl Q} B$, where the subscripts are used to denote the partial orders induced by $\cl P$ and $\cl Q$, respectively. Equivalently,  $\cl P$ is stronger than  $\cl Q$ if and only if the identity map on V is completely positive from $(V, \{P_n\}_{n=1}^\infty)$ to $(V, \{Q_n\}_{n=1}^\infty)$.\\

\begin{defn} Let $(V,\{C_n\}_{n=1}^\infty)$ be a matrix ordered $*$-vector space with matrix order unit $e$. A linear map $\phi:V\to M_k$ is called {\bf unital} if $\phi(e)=I_k$, $\phi$ is called {\bf positive} if $\phi(V^+) \subseteq M_k^+$, and $\phi$ is called {\bf k-positive} if $\phi^{(k)}(C_k)\subseteq M_k(M_k)^+$. Set $$S_k(V)=\{\phi:V\to M_k\,|\, \phi \text{ unital k-positive maps }\}.$$
\end{defn}

\begin{prop}\label{kpos} Let $(V,\{C_n\}_{n=1}^\infty)$  be a matrix ordered $*$-vector space. For a fixed $k\in \bb N$, let $\phi:V \to M_k$ be a linear map. Then $\phi$ is completely positive if and only if $\phi$ is k-positive.
\end{prop}
\begin{proof} This is a known proposition~\cite{VIP}, but here we provide a different simple proof:\\
If $\phi$ is completely positive, then $\phi$ is k-positive for each $k\in \bb N$. Now assume $\phi$ is k-positive. Before showing $\phi^{(n)}(v)\ge 0$ for all $v\in C_n, n\ge k$, we will prove the following result:\\ Given any vector $x\in \bb C^n\otimes \bb C^k$, there exists an isometry $\beta:\bb C^k\to \bb C^n$ and a vector $\tilde x\in \bb C^k\otimes \bb C^k$ such that $(\beta\otimes I_k)(\tilde x)=x$ for all $n\ge k$ in $\bb N$. For this, let $e_i=e_i^{(k)}=(0,\dots,0,1_i,0,\dots,0)$ be the usual basis vectors for $\bb C^k$, and let  $x\in \bb C^n\otimes \bb C^k$. Then there exist unique vectors $x_i\in \bb C^n,\, i=1,2,\dots,k$ with $x=\sum_{i=1}^k x_i\otimes e_i^{(k)}.$ Let $\cl F\subseteq \bb C^n$ be the subspace spanned by the vectors $x_i$. Then we have dim$\cl F\le k\le n$. Thus, we may find an isometry $\beta: \bb C^k\to \bb C^n$ whose range contains $\cl F$. For each $i$, we have a unique vector $\tilde x_i\in \bb C^k$  such that $\beta(\tilde x_i)=x_i$. Thus, if $\tilde x=\sum_{i=1}^k\tilde x_i\otimes e_i^{(k)}$, then $(\beta\otimes I_k)(\tilde x)=x$. Now, let $v\in C_n$ and $n\ge k$. Then \begin{eqnarray*}\langle \phi^{(n)}(v)x, x\rangle &=&\langle \phi^{(n)}(v)(\beta\otimes I_k)(\tilde x),(\beta\otimes I_k)(\tilde x)\rangle\\ &=&\langle(\beta^*\otimes I_k) \phi^{(n)}(v)(\beta\otimes I_k)(\tilde x),(\tilde x)\rangle\\ &=& \langle \phi^{(k)}(\beta^*v\beta)\tilde x, \tilde x\rangle \ge 0. \end{eqnarray*}Thus, $\phi$ is n-positive for all $n\in \bb N$, i.e. completely positive.
\end{proof}

\begin{remark}Let $(V,\{C_n\}_{n=1}^\infty)$ be a matrix ordered $*$-vector space with matrix order unit $e$, and let $\phi:V\to M_k$ be a linear map.
\begin{itemize}
\item[(i)] We can think of $\phi$ as a $k\times k$ matrix of linear functionals $\phi_{ij}:V\to \bb C$, i.e. $$\phi=\left[\phi_{ij}\right]_{i,j=1}^k\in M_k(V').$$
\item[(ii)] If $\phi$ is a positive linear map, then $\phi(v^*)=\overline{\phi(v)}^t$ for all $v\in V$, where $t$ stands for the transpose.
\item[(iii)] If $\phi$ is a positive linear map, then all the diagonal entries of $\phi$ are positive linear functionals. Moreover, if $\phi$ is unital, then the diagonal entries are states.
\item[(iv)] If $\phi$ is a positive linear map such that $\phi(e)= D_r \oplus 0,\, 1\le r \le k$, where $$D_r=\left[\begin{matrix}{\begin{matrix}d_1 & \\  & d_2 \end{matrix}} & {\text{\emph{\LARGE{0}}}}\\{\text{\emph{\LARGE{0}}}}&{\begin{matrix} \ddots & \\  & d_r\end{matrix}}\end{matrix}\right], \, d_i\in \bb R^+,\, 1\le i \le r,$$ then one can easily verify that $\phi$ can be written as $\phi=\tilde\phi\oplus 0,$ where $\tilde\phi:V\to M_r$ is a positive map such that $\tilde\phi(e)=D_r$. In this case, $\phi$ is called to be {\bf a positive diagonal map of rank r}, $1\le r\le k$.
\end{itemize}
\end{remark}

\begin{lemma}\label{ued} Let $(V,V^+)$ be an ordered $*$-vector space with order unit $e$. If $\phi:V\to M_k$ is a non-zero positive map, then $\phi$ is unitarily equivalent to a positive diagonal map $\psi:V\to M_k$ of rank $r\le k$.
\end{lemma}

\begin{proof} Let $\phi:V\to M_k$ be a positive map such that $\phi(e)=P\in M_k^+$. The rank of the matrix $P$ is at least $1$ and at most $k$. Without loss of generality, assume $\text{rank}(P)=r$, for some $1\le r\le k$. There exists a unitary $U$, such that $U^*PU=D_r \oplus 0$, where $D_r$ is an $r\times r$ diagonal matrix with positive diagonal entries. Define $\psi:V\to M_k$ by $$\psi(\cdot)=U^*\phi(\cdot)U.$$ It is straightforward to check that $\psi$ is a positive linear map with $\psi(e)=D_r\oplus 0$, i.e. $\psi$ is a rank $r$ positive map, $1\le r\le k$. Hence, $\phi$ is unitarily equivalent to such a map.
\end{proof}
\begin{cor}\label{kued} Let $(V,\{C_n\}_{n=1}^\infty)$ be a matrix ordered $*$-vector space with matrix order unit $e$. If $\phi:V\to M_k$ is a non-zero k-positive map, then $\phi$ is unitarily equivalent to a k-positive diagonal map $\psi:V\to M_k$ of rank $r\le k$.
\end{cor}
\begin{remark}\label{cum} If $(V,\{C_n\}_{n=1}^\infty)$ is a matrix ordered $*$-vector space with matrix order unit $e$ and $\phi:V\to M_k$ is a non-zero k-positive map with $\phi(e)=P\ge 0$ of rank $r\le k$, then one can easily show that $\phi$ is congruent to some k-positive map $\psi\oplus 0:V\to M_k$ with $\psi\in S_r(V)$ by using Lemma~\ref{ued} and Corollary~\ref{kued}.
\end{remark}

\noindent The following proposition is a generalization of Proposition~3.12 and Proposition~3.13 encountered in~\cite{PT}:
\begin{prop}\label{v=0} Let $(V,\{C_n\}^\infty_{n=1})$ be a matrix ordered $*$-vector space with matrix order unit $e$ such that $(V, C_1=V^+, e)$ is an AOU space. If $v\in V$ and $\phi(v) \ge 0$ for each $\phi\in S_k(V)$, then $v\in V^+$. Furthermore, if $\phi(v)=0$ for all such $\phi$, then $v=0$.
\end{prop}
\begin{proof} Let $s:V\to \bb C$ be a state on $V$, i.e. $s\in S(V)$. Define $$\phi =I_k\otimes s=\left[\begin{matrix} {\begin{matrix}s &\\ & s\end{matrix}} & \text{\LARGE{0}} \\ \text{\LARGE{0}} & {\begin{matrix}\ddots & \\ & s\end{matrix}}\end{matrix}\right]_{k\times k}:V\to M_k.$$ Then $\phi \in S_k(V)$. Let $v\in V$. Then $\phi(v)\ge 0$ if and only if $s(v)\ge 0$. This implies $v\in V^+$. Moreover, $\phi(v)=0$ if only if $s(v)=0$, which implies $v=0$.(more details in ~\cite{PT}.)
\end{proof}

Let $(V,V^+)$ be an ordered $*$-vector space with order unit $e$. We endow the real subspace $V_{sa}$ with the so-called {\bf order seminorm} $\|v\| =\inf\{r| -re \le v\le re \}$. We extend this order seminorm on  $V_{sa}$ to a {\bf $*$-seminorm} on V that preserves the $*$-operation, i.e. $\|v^*\|=\|v\|$  for all $v\in V$. We define {\bf the order seminorm on V} to be a $*$-seminorm $|||\cdot|||$ on $V$ with the property that $|||v|||=\|v\|$ for all $v\in V_{sa}$. If $e$ is an Archimedean order unit, then all these order seminorms become norms because $|||v|||=0$ implies $v\le 0$ and $v\ge 0$. Every order seminorm $\|\cdot\|$ on $V$ induces an {\bf order topology} on V, the topology with a basis consisting of balls $B_\epsilon (v)=\{w\in V: \|v-w\|< \epsilon\}$ for $v\in V$ and $\epsilon >0$. Note that since $\|\cdot\|$ is not necessarily a norm, this topology is not necessarily Hausdorff.

\begin{remark}Let $A\in M_k$ be a $k\times k$ matrix. Recall the {\it usual matrix norm} $$\|A\|=\sup\{\|Ax\|: x \in \bb C^k \text{ with } \|x\|\le 1\}.$$ From matrix theory~\cite{HJ}, we know that if we divide $A$ into block matrices $$A =\left[\begin{matrix}A_r & *\\ * & *\end{matrix}\right]\in M_k,$$ where $A_r\in M_r, \, 1\le r\le k$, then $\|A\|\ge\|A_r\|$. Moreover, if $A=A_r\oplus 0$, then the norms are the same, i.e. $\|A\|=\|A_r\|$.
\end{remark}

\begin{defn} Let $(V,\{C_n\}_{n=1}^\infty)$ be a matrix ordered $*$-vector space with matrix order unit $e$. We define the {\bf k-minimal order seminorm} $\|\cdot\|_{k-min}:V\to [0,+\infty)$ by $$\|v\|_{k-min}=\sup\big\{\|\phi(v)\|\, : \phi\in S_k(V)\big\}.$$
\end{defn}

\noindent{\bf Note:} When k=1, the k-minimal order seminorm becomes the usual minimal order seminorm defined in ~\cite{PT} by $\|v\|_m = \sup\{|s(v)|: s \, \text{ is a state}\}$. And $\|\cdot\|_m \le |||\cdot|||$ for every other $*$-seminorm $ |||\cdot|||$ on $V$.\\ By definition, we have $\|e\|_{k-min} =\|e\|_m=|||e||| =1$. If $(V,\{C_n\}_{n=1}^\infty,e)$ is an abstract operator system, and $\phi:V\to M_k$ is k-positive such that the norm of $\|\phi(e)\| \le 1$ with respect to k-minimal norm, then $\phi$ is called a {\bf contraction}.

\begin{prop}  Let $(V,\{C_n\}_{n=1}^\infty)$ be a matrix ordered $*$-vector space with a matrix order unit $e$ and let $k\in \bb N$. Then  $$\|v\|_{k-min} =\sup\big\{\|\phi(v)\|\, : \phi\in \bigcup_{r=1}^k S_r(V)\,\big\}.$$
\end{prop}
\begin{proof} For fixed $k\in \bb N$, let $r\le k$. If $r=k$, then it is clear that $\|v\|_{k-min}=\|v\|_{r-min}$. Assume $r <k$ and let $\phi\in S_k(V)$. Write $$\phi=\left[\phi_{ij}\right]_{i,j=1}^k=\left[\begin{matrix} \left[\phi_{ij}\right]_{i,j=1}^r & *\\ * & *\end{matrix}\right]. $$ Denote $\left[\phi_{ij}\right]_{i,j=1}^r =\phi_r$. Then, one can easily verify that $\phi_r \in S_r(V)$. Hence, we have
$$\|\phi(v)\|=\left\|\left[\begin{matrix}\phi_r(v) & *\\ * & *\end{matrix}\right]\right\| \ge \|\phi_r(v)\|.$$ By taking supremum over all $\phi\in S_k(V)$, we obtain $$\|v\|_{k-min}\ge \|\phi_r(v)\|,\, \text{ for all }\phi_r \in S_r(V).$$ This implies $\|v\|_{k-min}\ge \|v\|_{r-min}$. As a result, we conclude that $$\|v\|_{k-min} =\sup\big\{\|\phi(v)\|\, : \phi\in \bigcup_{r=1}^k S_r(V)\,\big\}.$$
\end{proof}

\begin{thm} Let $(V,\{C_n\}_{n=1}^\infty)$ be a matrix ordered $*$-vector space with matrix order unit $e$ such that $(V,V^+,e)$ is an AOU space. Let $|||\cdot|||$ be any order norm on V such that $\|\cdot\|_{k-min}\le |||\cdot|||$ and let $\phi:V\to M_k$ be a k-positive map. If $\|\phi\|$ denotes the norm of the k-positive map $\phi$ with respect to the order norm $|||\cdot|||$, then $\|\phi\|=\|\phi(e)\|_{M_k}$. Moreover, if $\phi$ is unital, then $\|\phi\|=1$.
\end{thm}
\begin{proof} By Lemma \ref{ued} above, we have that any positive map $\phi:V\to M_k$ is unitarily equivalent to a rank $r\le k$ diagonal map $\psi: V\to M_k$ such that $\psi =(\tilde \psi) \oplus 0$, with $\tilde \psi(e)=D_r$, for all $1\le r\le k$. Therefore, $\|\phi\|=\|\psi\|$. Note that $\|\psi\|=\|\tilde \psi\|$. Hence, it's enough to show that $\|\phi\|=\|\phi(e)\|$ for any diagonal map $\phi$ of rank k.
Let $\phi:V\to M_k$ be a k-positive map with $\phi(e)=D_k \ge 0$ invertible. Then $\psi= \phi(e)^{-1/2}\phi\phi(e)^{-1/2}$ is a unital k-positive map, and for any $v\in V$, we have
\begin{eqnarray*}\|\phi(v)\|& =&\|\phi(e)^{1/2}\psi(v)\phi(e)^{1/2}\|\le \|\phi(e)\|^{1/2} \cdot\|\psi(v)\|\cdot \|\phi(e)\|^{1/2}\\
&\le&\|\phi(e)\|\cdot \sup\{\|\varphi(v)\|: \varphi \in S_k(V) \}\\
&= & \|\phi(e)\|\cdot\|v\|_{k-min}\le \|\phi(e)\|\cdot|||v|||.
\end{eqnarray*}
So, we have $\|\phi\|\le \|\phi(e)\|$. In addition, since $\|e\|_{k-min}=|||e|||=1$, it follows that $\|\phi\|=\|\phi(e)\|$. Moreover, if $\phi$ is unital, then $\|\phi\|=1$.
\end{proof}

\indent We denote by $B(\cl H)$ the space of all bounded linear operators acting on a Hilbert space $\cl H$. A {\bf concrete operator system} $\cl S$ is a subspace of $B(\cl H)$ such that $\cl S=\cl S^*$ and $I_{\cl H}\in \cl S$. As is the case for many classes of subspaces (and subalgebras) of $B(\cl H)$, there is an abstract characterization of concrete operator systems, as was shown in~\cite{PTT}. If $\cl S\subseteq B(\cl H)$ is a concrete operator system, then we observe that $\cl S$ is a $*$-vector space, $\cl S$ inherits an order structure from $B(\cl H)$, and has $I_{\cl H}$ as an Archimedean order unit. Moreover, since $\cl S\subseteq B(\cl H)$, we have that $M_n(\cl S)\subseteq M_n(B(\cl H))\cong B(\cl H^n)$ and hence $M_n(\cl S)$ inherits a natural order structure from $B(\cl H^n)$ and the $n\times n$ diagonal matrix $$I_n\otimes I_{\cl H}=\left[\begin{matrix}{\begin{matrix}I_{\cl H} & \\ & I_{\cl H}\end{matrix}} & \text{\LARGE{0}}\\ \text{\LARGE{0}} &{\begin{matrix}\ddots & \\ & I_{\cl H}\end{matrix}}\end{matrix}\right]$$ is an Archimedean order unit for $M_n(\cl S)$. In other words, $\cl S$ is an abstract operator system $(\cl S, \{M_n(\cl S)^+\}_{n=1}^\infty, I_{\cl H})$, where each matrix cone $M_n(\cl S)^+$ contains $n\times n$ positive matrices in $M_n(B(\cl H))$ for all $n\in \bb N$. We will call this matrix ordering $\{M_n(\cl S)^+\}_{n=1}^\infty$ as {\bf the natural operator system structure} of $\cl S$ inherited by the order structure of $B(\cl H)$. The following result of Choi and Effros~\cite{CE, VIP} shows that the converse is also true.

\begin{thm}[Choi-Effros]\label{c-e} Every concrete operator system $\cl S$ is an (abstract) operator system. Conversely, if $(V, \{C_n\}_{n=1}^\infty,e)$ is an (abstract) operator system, then there exists a Hilbert space $\cl H$, a concrete operator system $\cl S\subseteq B(\cl H)$, and a complete order isomorphism $\phi:V\to \cl S$ with $\phi(e)=I_{\cl H}.$
\end{thm}

Thus, every operator system $\cl S\subseteq B(\cl H)$ can be viewed as a matrix ordered $*$-vector space $(\cl S, \{M_n(\cl S)^+\}_{n=1}^\infty)$ with (Archimedean) matrix order unit $e=I_{\cl H}$. Therefore for the rest of this paper, given an operator system $\cl S$, we will use its "natural operator system structure" inherited by the order structure of $B(\cl H)$ for some Hilbert space $\cl H$, and build new operator system structures on it.

\section{The Super k-Minimal Operator System Structures on $\cl S$}\label{oss}

Let $\cl S$ be an operator system and let $e$ be its unit. Before setting up the k-minimal operator system structure on the AOU space $(\cl S, \cl S^+, e)$, recall the weakest operator system structure, introduced in~\cite{PTT} and denoted by $\cl C^{min}(\cl S)=\{C^{min}_n(\cl S)\}_{n=1}^\infty$, where
\begin{eqnarray*}C^{min}_n(\cl S)&=&\big\{(a_{ij})\in M_n(\cl S): (s(a_{ij}))\in M_n^+,\text{ for all } \,s \in S(\cl S)\big\}\\ &=&\big\{(a_{ij})\in M_n(\cl S): (f(a_{ij}))\in M_n^+,\, f \text{ positive linear functional}\big\}\\ &=& \big\{(a_{ij})\in M_n(\cl S): \alpha^*(a_{ij})\alpha \in \cl S^+,\text{ for all }\alpha \in \bb C^n \big\}. \end{eqnarray*}

$\cl C^{min}(\cl S)$ is the operator system structure on $\cl S$, induced by the inclusion of $\cl S$ into $C(S(\cl S))$, the $C^*$-algebra of continuous funtions on $S(\cl S)$, set of states on $\cl S$. And OMIN$(\cl S)$ is the operator system $(\cl S,\cl C^{min}(\cl S), I_{\cl H})$, which can be identified as a subspace of $C(S(\cl S))$, up to complete order isomorphism.\\

In the next result, we generalize the complex version of Kadison's characterization of function systems~\cite{PT, VIP}:
\begin{thm}\label{krt} Let $\cl S$ be an operator system with unit $e$ and fix $k\in \bb Z^+$. Give $\cl S$ the order topology generated by the k-minimal order norm, denoted as $\cl S_{k-min}$, and endow the space of unital k-positive linear maps $S_k(\cl S)=\{\phi:\cl S\to M_k\, | \phi \text{ is unital k-positive map}\,\}$ with the corresponding weak$^*$-topology. Then $S_k(\cl S)$ is a compact space, and the map
$$\Gamma:\cl S\to M_k(C(S_k(\cl S)))\text{ given by } \Gamma(a)(\phi)=\phi(a)$$ is an injective map that is an order isomorphism onto its range with the property that $\Gamma(e)=I_k$. Furthermore, $\Gamma$ is an isometry with respect to the k-minimal order norm on $\cl S$ and the sup norm on $M_k(C(S_k(\cl S)))$.
\end{thm}
\begin{proof} Let $\cl S$ be a given operator system with unit $e$. Then $(\cl S,\cl S^+,e)$ is an AOU space, and its dual $\cl S^*$ is a normed $*$-vector space. For fixed $k\in \bb N$, one can show that $M_k(\cl S^*)=\{\phi=(\phi_{ij}):\cl S\to M_k\, |\, \phi_{ij}\in \cl S^*\text{ for all }1\le i,j\le k\}$ is a normed $*$-vector space, too. Then the unit ball of $M_k(\cl S^*)$ is defined as
$$(M_k(\cl S^*))_1=\{\phi\in M_k(\cl S^*): \|\phi\|\le 1\}.$$ Endowing $\cl S$ with any order norm $|||\cdot|||$ makes $S_k(\cl S)$ a subset of the unit ball of $M_k(\cl S^*)$. In addition, suppose that $\{\phi_\lambda\}_{\lambda\in\Lambda}\subseteq S_k(\cl S)$ is a net of these maps, and $\lim \phi_\lambda=\phi$ in the weak$^*$-topology for some $\phi\in M_k(\cl S^*)$. Then for any $a\in \cl S^+$ we have that $\lim \phi_\lambda(a)=\phi(a)$ , and since $\phi_\lambda(a)\ge 0$  for all $\lambda$, it follows that $\phi(a)\ge 0$ for all $a\in \cl S^+$. Similarly, for any $A\in M_k(\cl S)^+$, we have that $\lim \phi^{(k)}_\lambda(A)=\phi^{(k)}(A)$ , and since $\phi^{(k)}_\lambda(A)\ge 0$ for all $\lambda$, it follows that $\phi^{(k)}(A)\ge 0$ for all $A\in M_k(\cl S)^+$. Hence $\phi$ is a k-positive linear map. Moreover, $\phi(e)=\lim \phi_\lambda(e)=\lim I_k=I_k$, i.e. $\phi$ is unital. Thus $S_k(\cl S)$ is closed in the weak$^*$-topology.\\ \indent In the case of the k-minimal order norm, we have $S_k(\cl S)\subseteq (M_k(\cl S^*_{k-min}))_1$, and the latter $(M_k(\cl S^*_{k-min}))_1\cong (M_k(\cl S_{k-min}))^*_1$. It follows from Alaoglu's Theorem\cite[Theorem 3.1]{JBC}, that $(M_k(\cl S_{k-min}))^*_1$ is compact in the weak$^*$-topology, which implies that $(M_k(\cl S^*_{k-min}))_1$ is compact, too. Since $S_k(\cl S)$ is a closed subset of this compact ball, we have that $S_k(\cl S)$ is compact in the weak$^*$-topology.\\
\indent Consider the continuous matrix-valued functions $\hat a: S_k(\cl S) \to M_k$ given  by $\hat a(\phi)=\phi(a) \in M_k.$ The collection of such continuous functions $\{\hat a: S_k(\cl S)\to M_k \}$ together with $\|\cdot\|_{k-min}$ norm, form the unital $C^*$-algebra $M_k(C(S_k(\cl S)))$, i.e. $$M_k(C(S_k(\cl S))) \equiv \big\{\hat a: S_k(\cl S)\to M_k | \hat a \text{ continuous matrix-valued function } \big\}.$$ Let $\Gamma: \cl S\to M_k(C(S_k(\cl S)))$ be the map given by $\Gamma(a)(\phi)=\phi(a)$. If $\Gamma(a)=0$ for some $a\in \cl S$, then $\phi(a)=0$ for all $\phi \in S_k(\cl S)$. It follows from Proposition~\ref{v=0} that $a=0$. Therefore, $\Gamma$ is one-to-one.\\
\indent In addition, if $a\in \cl S^+$, then  for any $\phi\in S_k(\cl S)$ we have that $\Gamma(a)(\phi)=\phi(a)\in M_k^+$ by the positivity of $\phi$. Hence the function $\Gamma(a)$ takes on nonnegative values and $\Gamma(a)\in M_k(C(S_k(\cl S)))^+$. Conversely, if $\Gamma(a) \in M_k(C(S_k(\cl S)))^+$, then for all $\phi\in S_k(\cl S)$ we have that $\phi(a)= \Gamma(a)(\phi) \ge 0$. This implies $a\in \cl S^+$ by Proposition~\ref{v=0}. Therefore, $\Gamma$ is an order isomorphism onto its range. Finally, if $a\in \cl S$, then
\begin{eqnarray*} \|a\|_{k-min}&=&\sup\{\|\phi(a)\|\,|\, \phi \in S_k(\cl S)\,\}\\
&=&\sup\{\|\Gamma(a)(\phi)\|: \phi \in S_k(\cl S)\}\\
&=&\|\Gamma(a)\|_\infty.\end{eqnarray*}
so that $\Gamma$ is an isometry with respect to the k-minimal order norm on $\cl S$ and the sup norm on $M_k(C(S_k(\cl S)))$.
\end{proof}

\begin{remark} Since unital $C^*$-algebras are operator systems, the order isomorphism map $\Gamma$ of Kadison's Representation Theorem induces a new operator system structure $\{C_n\}_{n=1}^\infty$ on $\cl S$. We have $C_1= \cl S^+= \Gamma^{-1}(P_1)$, where $P_1$ denotes the set of nonnegative matrix-valued continuous functions on $S_k(\cl S)$. In addition, we say $(a_{ij})\in C_n$ if and only if $(\Gamma(a_{ij})) \in M_n(M_k(C(S_k(\cl S))))^+$, if and only if $(\phi(a_{ij}))\in M_{nk}^+$ for every $\phi\in S_k(\cl S)$ .
\end{remark}

\begin{defn} Let $\cl S$ be an operator system with unit $e$. For each $n \in \bb N$ set $$C_n^{k-min}(\cl S)=\Big\{(a_{ij})\in M_n(\cl S): (\phi(a_{ij})) \ge 0, \text{ for all } \phi \in S_k(\cl S) \Big\}, $$  $\cl C^{k-min}(\cl S)=\{C_n^{k-min}(\cl S)\}_{n=1}^\infty$ and define {\emph{OMIN}}$_k(\cl S)=(\cl S, \cl C^{k-min}(\cl S),e)$.
\end{defn}

\noindent By the definition and the remark above, $\cl C^{k-min}(\cl S)$ is the operator system structure on $\cl S$ induced by the inclusion of $\cl S$ into $M_k(C(S_k(\cl S)))$. We call $\cl C^{k-min}(\cl S)$ {\bf the super k-minimal operator system structure} on $\cl S$, and we call OMIN$_k(\cl S)$ {\bf the super k-minimal operator system}. \\\\\\

\noindent{\bf Properties of super k-minimal operator system structures on $\cl S$:}\\

\noindent $(1)$ When $k=1$, $C_n^{1-min}(\cl S)=C_n^{min}(\cl S)$ for all $n\in \bb N$:
\begin{eqnarray*}\,\,\, (a_{ij}) \in C_n^{1-min}(\cl S) & \iff & (s(a_{ij})) \in M_n^+, \text{ for all } s \in S_1(\cl S)=S(\cl S),\\ &\iff & \alpha^* (a_{ij})\alpha \in \cl S^+, \text{ for all } \alpha \in \bb C^n.
\end{eqnarray*}

\noindent $(2)$ $M_n(\cl S)^+\subseteq C_n^{k-min}(\cl S)\subseteq C_n^{min}(\cl S)$, for all $n \in \bb Z^+$:\\\\ Note that $M_n(\cl S)^+\subseteq C_n^{k-min}(\cl S)$ is obvious by the definition of $C_n^{k-min}(\cl S)$. Now, let $(a_{ij})\in C_n^{k-min}(\cl S)$ for some fixed $k\in \bb Z^+$, and let $\alpha \in \bb C^n$. Then
$$ \,\,\,\, 0 \le  (\alpha^* \otimes I_k)(\phi(a_{ij}))(\alpha \otimes I_k)  = \phi( \alpha^*[a_{ij}]\alpha),\text{ for all }\phi \in S_k(\cl S).$$ This implies $\alpha^*(a_{ij})\alpha \in \cl S^+$, for all $\alpha \in \bb C^n$, i.e. $(a_{ij}) \in C_n^{min}(\cl S)$.\\\\
\noindent $(3)$ $C_n^{h-min}(\cl S) \subseteq C_n^{k-min}(\cl S)$ for all $h \ge k$:\\\\ Let $(a_{ij}) \in C_n^{h-min}(V)$. The equality holds when $h=k$. Suppose $h>k$ and let $\phi\in S_k(\cl S)$ and $s\in S(\cl S)$. Define $\Phi:\cl S\to M_h$ by  $$\Phi =\phi \oplus \underbrace{s \oplus \dots \oplus s}_{\text{(h-k) times}}=\left[\begin{matrix}{\begin{matrix}\phi & \\ & s\end{matrix}} & \text{\LARGE{0}}\\ \text{\LARGE{0}}& {\begin{matrix}\ddots & \\ & s\end{matrix}}\end{matrix}\right].$$  One can easily verify that $\Phi$ is a well-defined positive linear function with $\Phi(e)=I_h$, i.e. $\Phi \in S_h(\cl S)$. This implies $(\Phi(a_{ij})) \ge 0$. Thus, we have:
 $$0\le (\Phi(a_{ij})) = \left[\begin{matrix}{\begin{matrix}\phi(a_{ij}) & \\ & s(a_{ij})\end{matrix}} & \text{\LARGE{0}}\\ \text{\LARGE{0}}& {\begin{matrix}\ddots & \\ & s(a_{ij})\end{matrix}}\end{matrix}\right]_{i,j}.$$ By the canonical reshuffling, we obtain:
\begin{eqnarray*} 0\le(\Phi(a_{ij})) & \iff & {\left[\begin{matrix}{\begin{matrix}(\phi(a_{ij})) & \\ & (s(a_{ij}))\end{matrix}} & \text{\LARGE{0}}\\ \text{\LARGE{0}}& {\begin{matrix}\ddots & \\ & (s(a_{ij}))\end{matrix}}\end{matrix}\right]\ge 0}  \\ \\ & \iff& (\phi(a_{ij})) \ge 0, \text{ for all }\phi \in S_k(\cl S)\\ &\iff & (a_{ij}) \in C^{k-min}_n(\cl S).
\end{eqnarray*}\\
\noindent $(4)$ The identity map $\imath : \text{OMIN}_h(\cl S)\to \text{OMIN}_k(\cl S)$ is completely positive, whenever  $h\ge k$. \\

\begin{prop} Let $\cl S$ be an operator system with unit $e$, $f:\cl S\to M_k$ be a k-positive linear map, and $$C^k_n(\cl S) =\big\{(a_{ij})\in M_n(\cl S)|\, (f(a_{ij})) \in M_{nk}^+,\,  f: \cl S\to M_k \text{\emph{ k-positive}} \big\}.$$ Then $\{C^k_n(\cl S)\}_{n=1}^\infty$ is the super k-minimal operator system structure on $\cl S$.
\end{prop}
\begin{proof} It suffices to show that $C^k_n(\cl S)=C^{k-min}_n(\cl S)$ for all n. One can see that $C^k_n(\cl S)\subseteq C^{k-min}_n(\cl S)$ is trivial, since $S_k(\cl S)$ is just a subset of all k-positive linear maps from $\cl S$ to $M_k$. On the other hand, let $(a_{ij})\in C^{k-min}_n(\cl S)$ and let $\phi:\cl S\to M_k$ be a k-positive map with $\phi(e)=P\ge 0$. Then there exists a unital k-positive map $\psi\in S_k(\cl S)$ such that $\phi(\cdot)=P^{1/2}\psi(\cdot)P^{1/2}$ \cite[Exercise 6.2]{VIP}. Hence, we have $$(\phi(a_{ij}))=(P^{1/2}\psi(a_{ij})P^{1/2})=(I_k\otimes P^{1/2})(\psi(a_{ij})(I_k\otimes P^{1/2}).$$ This shows that $(a_{ij})\in C_n^k(\cl S)$ and $C^{k-min}_n(\cl S) \subseteq C^k_n(\cl S)$. Hence, $C^k_n(\cl S)= C_n^{k-min}(\cl S)$ is the super k-minimal operator system structure on $\cl S$.
\end{proof}

\begin{remark} The above result shows that we can define the super k-minimal operator system structure in a more general way, as $$C_n^{k-min}(\cl S)= \big\{(a_{ij})\in M_n(\cl S)|\, (\phi(a_{ij}))\ge 0, \, \phi:\cl S\to M_k \text{ k-positive map } \big\}.$$
\end{remark}

\begin{lemma}\label{kposx} Let $\cl S$  be an operator system and let $X$ be a compact space. If $\psi: \cl S\to M_k(C(X))$ is k-positive, then $\psi$ is completely positive.
\end{lemma}
\begin{proof} Define $\pi_x  : M_k(C(X))\to M_k$ to be the point-evaluation matrix function, i.e. $\pi_x((f_{ij}))=(f_{ij}(x))$. It is clear that $\pi_x$ is a well-defined $*$-homomorphism. Moreover, $\pi_x$ is completely positive. Consider $\pi_x \circ \psi :\cl S\to M_k$.  Let $(a_{ij})\in M_k(\cl S)^+$, then $(\psi(a_{ij})) \in M_k(M_k(C(X)))^+$, which implies $(\pi_x(\psi(a_{ij}))) \in M_{k^2}^+$, since $\pi_x$ is a completely positive map. The k-positivity of $\psi$ implies $\pi_x \circ \psi: \cl S\to M_k$ is a k-positive map, and therefore completely positive by Proposition~\ref{kpos}. As a result, $\psi$ is completely positive.
\end{proof}

\begin{thm} Let $\cl S$ be an operator system with unit $e$. If  $(W,\{C_n\}_{n=1}^\infty)$ is  a matrix ordered $*$-vector space and $\phi: W \to \text{\emph{OMIN}}_k(\cl S)$ is k-positive, then $\phi$ is completely positive. \\ Moreover, if $\tilde{\cl S} = (\cl S,\{C_n\}_{n=1}^\infty,e)$ is another operator system on $\cl S$ with $C_k=M_k(\cl S)^+$ such that for every operator system $W$, any k-positive map $\psi: W\to \tilde{\cl S}$ is completely positive, then the identity map on $\cl S$ is a complete order isomorphism from $\tilde{\cl S}$ onto {\emph{OMIN}}$_k(\cl S)$.
\end{thm}
\begin{proof} $(i)$ It is clear that, up to complete order isomorphism, OMIN$_k(\cl S)$ can be identified with a subspace of $M_k(C(S_k(\cl S)))$. We know that $S_k(\cl S)$ is a compact space. Substituting $X= S_k(\cl S)$ in Lemma~\ref{kposx}, we get $\phi: W\to M_k(C(S_k(\cl S)))$ is completely positive, i.e.  $\phi: W\to \text{OMIN}_k(\cl S)$ is completely positive.\\
$(ii)$ Now, let $\tilde{\cl S} = (\cl S,\{C_n\}_{n=1}^\infty,e)$ be another operator system with $C_k=M_k(\cl S)^+$ such that for every operator system $W$, any k-positive map $\psi: W\to \tilde{\cl S}$ is completely positive. Note that $C_k=M_k(\cl S)^+$ implies that $C_i=M_i(\cl S)^+$ for all $i=1,2,\dots,k$. Moreover, we know that $M_n(\cl S)^+\subseteq C_n^{k-min}(\cl S)),$ with equality holding for $1\le n\le k$. Hence, $$M_k(\cl S)^+= M_k(\tilde{\cl S})^+=C_k^{k-min}(\cl S). \text{  (*)}$$ Let $\imath: \tilde{\cl S}\to \text{OMIN}_k(\cl S)$ be the identity map on $\cl S$. By the identity $(*)$ above, both $\imath$ and $\imath^{-1}$ are k-positive maps. Then, by $(i)$, $\imath$ is completely positive, and by the assumption, $\imath^{-1}$ is completely positive. Since $\imath$ is also unital, we have that $\tilde{\cl S}$ and OMIN$_k(\cl S)$ are completely order isomorphic. 

\end{proof}

\begin{prop}Let $\cl S$ be an operator system with unit $e$ and fix $k\in \bb N$. Then the identity map $\text{\emph{id}}$ from \emph{OMIN}$_k(\cl S)$ to $\cl S$  is completely bounded with $\|\text{\emph{id}}\|_{cb}=C$ if and only if for every operator system $\cl T$, every unital k-positive map $\phi:\cl T\to \cl S$ is completely bounded and the supremum of the completely bounded norms of all such maps is $C$.
\end{prop}
\begin{proof}Refer to \cite[Proposition 5.3]{PTT}.
\end{proof}

\section{The Super k-Maximal Operator System Structures on $\cl S$}\label{kmax}

Let $\cl S$ be an operator system with unit $e$. For all $n\in \bb N$, we have that $M_n(\cl S)=M_n\otimes \cl S$, the natural algebraic tensor product. Moreover, we have that $M_n(\cl S)_{sa} =(M_n)_{sa}\otimes \cl S_{sa}$, where the right-hand side is the algebraic tensor of real vector spaces.\\
\indent Recall the strongest matrix ordering $\cl D^{max}(\cl S)=\{D_n^{max}(\cl S)\}_{n=1}^\infty$ on the AOU space $(\cl S,\cl S^+,e)$, where each matrix cone $D_n^{max}(\cl S)$ is given by
\begin{eqnarray*} D_n^{max}(\cl S) &=&\bigg\{\sum_{i=1}^k a_i\otimes s_i: s_i\in \cl S^+, a_i \in M_n^+, 1\le i\le k, k \in \bb N \bigg\}\\
&=&\bigg\{ A\, \text{diag}(s_1,\dots,s_m)\, A^*: A \in M_{n,m}, s_i \in \cl S^+, m\in \bb N\bigg\},\end{eqnarray*} with $e$ being just a matrix order unit for this ordering on a general operator system $\cl S$, as was shown in~\cite{PTT}.

\begin{defn} Let $\cl S$ be an operator system with unit $e$. For some fixed $k\in \bb N$, set
\begin{multline*} D_n^{k-max}(\cl S)=\big\{ADA^*\,|\, A \in M_{n,mk},\, D = \text{diag}(D_1,\dots,D_m),\,\text{\emph{ where }}\\ D_l \in M_k(\cl S)^+,\, 1\le l\le m,\, m\in \bb N\, \big\}\end{multline*} and $\cl D^{k-max}(\cl S)=\{D_n^{k-max}(\cl S)\}_{n=1}^\infty$.
\end{defn}

\begin{prop}\label{kmaxo} Let $\cl S$ be an operator system with unit $e$. Then  $\cl D^{k-max}(\cl S)$ is a matrix ordering on $\cl S$ and $e$ is a matrix order unit for this ordering. In particular, $\cl D^{1-max}(\cl S)$ is the strongest matrix ordering on $\cl S$.
\end{prop}
\begin{proof} Need to check the three conditions of being a matrix ordering on $\cl S$:
\begin{itemize}
\item[(1)] $D_n^{k-max}(\cl S)$ is a cone in $M_n(\cl S)_{sa}$ for each $n\in \bb N$, and in particular, $D_1^{k-max}(\cl S)= \cl S^+$: \\ For each $n\in \bb N,\quad D_n^{k-max}(\cl S)$ is a non-empty subset of $M_n(\cl S)_{sa}$ as one can easily verify that $D_n^{max}(\cl S)\subseteq D_n^{k-max}(\cl S)$. Moreover, by definition we have that $D_n^{k-max}(\cl S)\subseteq M_n(\cl S)^+$ with the following two properties:
\begin{itemize}
\item[(i)] $D_n^{k-max}(\cl S)$ is closed under positive scalar multiplication:\\
Let $\lambda \in \bb R^+$ and $ADA^*\in D_n^{k-max}(\cl S)$, then
\begin{eqnarray*} \lambda(ADA^*)& =&(\sqrt{\lambda}A)D(\sqrt{\lambda}A)^*\\
&=& A(\lambda)A^* \in D_n^{k-max}(\cl S).\end{eqnarray*}
\item[(ii)] $D_n^{k-max}(\cl S)$ is closed under addition:\\
Let $ADA^*, B\tilde D B^* \in D_n^{k-max}(\cl S)$, where $A\in M_{n,mk},\, B\in M_{n,pk},$\\ $D=\text{diag}(D_1,\dots,D_m),\, \tilde D=\text{diag}(\tilde D_1,\dots,\tilde D_p),$  then we have:\\
$ADA^* +B{\tilde D}B^* =\left[\begin{matrix}A & B\end{matrix}\right]\left[\begin{matrix}D& 0\\ 0 & \tilde D\end{matrix}\right]\left[\begin{matrix}A^*\\ B^*\end{matrix}\right]\in D_n^{k-max}(\cl S). $
\end{itemize}
For $n=1$, we have $\cl S^+=D_1^{max}(\cl S)\subseteq D_1^{k-max}(\cl S)\subseteq \cl S^+$, i.e. $D_1^{k-max}(\cl S)=\cl S^+$.
\item[(2)] $D_n^{k-max}(\cl S) \cap -D_n^{k-max}(\cl S) =\{0\}$ for all $n\in \bb N$:\\
 Note that $D_n^{k-max}(\cl S) \cap -D_n^{k-max}(\cl S)\subseteq M_n(\cl S)^+ \cap -M_n(\cl S)^+  =\{0\}$.
\item[(3)] $X\,D_n^{k-max}(\cl S)\,X^* \subseteq D_m^{k-max}(\cl S)$ for all $X\in M_{m,n}$ and for all $m,n\in \bb N$:\\
 Let $ADA^* \in D_n^{k-max}(\cl S),\, X\in M_{m,n}$ for any $m,n \in \bb N$,\\ then $X(ADA^*)X^*=(XA)D(XA)^* \in D_m^{k-max}(\cl S)$,\\ i.e. $XD_n^{k-max}(\cl S)X^* \subseteq D_m^{k-max}(\cl S)$ for all $m,n$.
\end{itemize}
Hence, $(1),\, (2) \text{ and } (3)$ show that $\cl D^{k-max}(\cl S)$ is a matrix ordering on $\cl S$. It remains to show that $e$ is a matrix order unit for this ordering. It is clear that $e$ is an (Archimedean) order unit for $D_1^{k-max}(\cl S)$ since $D_1^{k-max}(\cl S)=\cl S^+$. Since $\cl D^{max}(\cl S)$ is the strongest matrix ordering on $\cl S$, then we have $D_n^{max}(\cl S) \subseteq D_n^{k-max}(\cl S)$ for all $n\in \bb N$. We know $e_n$ is an order unit for $(M_n(\cl S), D_n^{max}(\cl S))$. It follows that $e_n$ is an order unit for $(M_n(\cl S),D_n^{k-max}(\cl S))$, i.e. $e$ is a matrix order unit for $\cl D^{k-max}(\cl S)$. As a result, $\cl D^{k-max}(\cl S)$ is a matrix ordering on $\cl S$. In particular, for $k=1$ we have that $\cl D^{1-max}(\cl S)=\cl D^{max}(\cl S)$ is the strongest matrix ordering on $\cl S$.
\end{proof}

\begin{remark}\label{arch} Given an operator system $\cl S$ with unit $e$, we have that $(\cl S,\,{\cl D}^{k-max}(\cl S), e)$ is a matrix ordered $*$-vector space for some fixed $k\in \bb N$. \\\\
$(1)$ If the operator system $\cl S$ is finite-dimensional, then one can verify that this matrix ordered $*$-vector space $(\cl S,\, {\cl D}^{k-max}(\cl S), e)$ is an operator system. \\\\
$(2)$ If $\cl S$ is an infinite-dimensional space, then ${\cl D}^{k-max}(\cl S)$ is just a matrix ordering, not an operator system structure. As an example, let $\cl S=C([0,1])$ be the vector space of complex-valued functions on the unit interval, with $\cl S^+$ the usual cone of positive functions and $e$ the constant function $1$.\\ Let $P(t)=\left[\begin{matrix}1& e^{2\pi it}\\ e^{-2\pi it}&1\end{matrix}\right]$ be a self-adjoint element in $M_2(C([0,1]))$.\\ Then $re_2 +P(t) =\left[\begin{matrix}1+r & e^{2\pi it}\\ e^{-2\pi it}&1+r \end{matrix}\right] \in D_2^{max}(C([0,1]))$ for every $r>0$, \\ but $P(t)\notin D_2^{max}(C([0,1]))$ as was shown in~\cite{PTT}.\\ This shows that $e=1$ can not be an Archimedean matrix order unit. \\ As a result, $(C([0,1]), \cl D^{max}(C([0,1])), 1)$ can not be an operator system. Furthermore, the inclusion $\cl D^{max}(C([0,1]))\subseteq \cl D^{k-max}(C([0,1])), \, k\in\bb N,$ implies that $\cl D^{k-max}(C([0,1]))$ is just a matrix ordering, too.\\\\
$(3)$ To transform the matrix ordered space $(C([0,1]), \cl D^{max}(C([0,1])), 1)$ and consequently $(C([0,1]), \cl D^{k-max}(C([0,1])), 1)$ into operator systems, we will use the {\bf Archimedeanization process} for matrix ordered spaces. This theory was developed in detail for ordered $*$-vector spaces in ~\cite{PT}, and generalized to matrix ordered spaces with a matrix order unit $e$ in~\cite{PTT}. 
\end{remark}
 
\noindent The Archimedeanized matrix ordered $*$-vector space $(\cl S, \cl D^{max}(\cl S),e)$  with underlying operator system $\cl S$, matrix ordering $\cl C^{max}(\cl S) =\{C_n^{max}(\cl S)\}_{n=1}^\infty$, given by $C_1^{max}(\cl S)=D^{max}_1=\cl S^+$ and
$$C^{max}_n(\cl S)=\bigg\{A\in M_n(\cl S): re_n +A \in D_n^{max}(\cl S)\text{ for all } r>0 \, \bigg\},$$ is the {\bf maximal operator system} OMAX$(\cl S)=(\cl S, \cl C^{max}(\cl S), e)$ in~\cite{PTT}.

\begin{defn} Let $\cl S$ be an operator system with unit $e$. We set $$C_n^{k-max}(\cl S)=\bigg\{ A\in M_n(\cl S):\, re_n +A \in D_n^{k-max}(\cl S)\,\text{ for all }r>0\, \bigg\},$$  $\cl C^{k-max}(\cl S)=\{C_n^{k-max}(\cl S)\}_{n=1}^\infty$ and define $\text{\emph{OMAX}}_k(\cl S)=(\cl S, \cl C^{k-max}(\cl S), e)$.
\end{defn}
\noindent By the definition and the results above, we have that the enlarged matrix ordering $\cl C^{k-max}(\cl S)$ is a new operator system structure on $\cl S$, which we shall call the {\bf super k-maximal operator system structure} on $\cl S$ and OMAX$_k(\cl S)$ the {\bf super k-maximal operator system} on $\cl S$.\\\\
\noindent{\bf Properties of super k-maximal operator system structures on $\cl S$:}\\\\
\noindent$(1)$ When $k=1$, $C_n^{1-max}(\cl S)=C_n^{max}(\cl S)$ for all $n\in \bb N$.\\\\
\noindent$(2)$ For a fixed $k\in\bb N$, $C_n^{max}(\cl S)\subseteq C_n^{k-max}(\cl S)$ for all $n\in \bb N$.\\\\
\noindent$(3)$ $C_n^{k-max}(\cl S)\subseteq C_n^{h-max}(\cl S)$ for some fixed $k,h\in\bb N$ with $k\le h$:
\begin{itemize}
\item[(i)] Let $(a_{ij})\in C_n^{k-max}(\cl S)$. If $k=h$, then $C_n^{k-max}(\cl S)=C_n^{h-max}(\cl S)$. \\ Suppose $k<h$. If $(a_{ij})\in D_n^{k-max}(\cl S)$, then $(a_{ij}) = ADA^*$ for some $A\in M_{n,mk}$ and $D=\text{diag}(D_1, D_2, \dots,D_m)$, where each $D_i \in M_k(\cl S)^+$ for all $1\le i \le m,\, m\in \bb N$. Write $A =\left[\begin{matrix}A_1 & A_2 & \cdots & A_m\end{matrix}\right]$, where each $A_i \in M_{n,k}, \, 1\le i \le m$. Transform the matrix $A$ into:
$$ \tilde A=\left[\begin{matrix} A_1\, 0\, & A_2\, 0\, & \cdots  &\, A_m \, 0\end{matrix}\right] \in M_{n,hk},$$ by adding $(h-k)$ columns of $0$ after each block $A_i$. Using the same trick,\\ transform the block diagonal matrix $D$ into a bigger block diagonal matrix\\ $\tilde D = \text{diag}(\tilde D_1, \tilde D_2, \dots, \tilde D_m)$, where each diagonal block $\tilde D_i$ is maximized by\\ adding a $(h-k)\times (h-k)$ diagonal block of $0$, i.e.
$$\tilde D_i =\left[\begin{matrix}D_i & 0 \\ 0 & 0\end{matrix}\right]\in C_h^{min}(\cl S).$$ Then $ (a_{ij}) = ADA^*=\tilde A \tilde D \tilde A^* \in D_n^{h-max}(\cl S)$ and $D_n^{k-max}(\cl S) \subseteq D_n^{h-max}(\cl S)$.
\item[(ii)] Let $(a_{ij}) \in C_n^{k-max}(\cl S)$. Then $ re_n +(a_{ij}) \in D_n^{k-max}(\cl S)$ for all $r>0$. Then by case (i), $re_n +(a_{ij}) \in D_n^{h-max}(\cl S)$ too. Therefore,  $(a_{ij}) \in C_n^{h-max}(\cl S)$ and  $C_n^{k-max}(\cl S)\subseteq C_n^{h-max}(\cl S)$. 
\end{itemize}
\noindent$(4)$ The identity map $\imath : \text{OMAX}_k(\cl S)\to \text{OMAX}_h(\cl S)$ is completely positive, whenever  $k\le h$.

\begin{lemma}\label{rfactor} Let $(W, W^+, e)$ be an AOU space, and let $\{P_n\}_{n=1}^\infty$ be an operator system structure on $W$ with $P_1 = W^+$. If $p\in W^+, \, (w_{ij})\in M_n(W)$ are such that $r(p\otimes I_n) + (w_{ij}) \in P_n$ for all $r>0$, then $(w_{ij})\in P_n$.
\end{lemma}
\begin{proof} Let $\|\cdot\|$ be an order seminorm for this structure. If $p=0$, then $(w_{ij})\in P_n$ is obvious. Let $0\ne p\in W^+$, and replace $p$ by $\frac{p}{\|p\|} \in W^+$. This implies $r(\frac{p}{\|p\|}\otimes I_n)+ (w_{ij})\in P_n$, too, for all $r>0$. Since $\|\frac{p}{\|p\|}\|=1$ and $e-\frac{p}{\|p\|} \in W^+$, then we have $$re_n +(w_{ij}) = r((e-\frac{p}{\|p\|})\otimes I_n) + r(\frac{p}{\|p\|} \otimes I_n)+(w_{ij}) \in P_n,$$ for all $r>0$. Therefore, $(w_{ij})\in P_n$.
\end{proof}

\begin{thm}\label{omaxk} Let $\cl S$ be an operator system with unit $e$ and $(W,\{P_n\}_{n=1}^\infty, e')$ be an (abstract) operator system. If $\phi: \text{\emph{OMAX}}_k(\cl S) \to W$ is a k-positive map for some fixed $k\in \bb N$, then $\phi: \text{\emph{OMAX}}_k(\cl S) \to W$ is completely positive. \\ Moreover, if $\tilde{\cl S} = (\cl S,\{C_n\}_{n=1}^\infty,e)$ is another operator system on $\cl S$ with $C_k=M_k(\cl S)^+$ such that for every operator system $W$, any k-positive map $\psi:\tilde{\cl S}\to W$ is completely positive, then the identity map on $\cl S$ is a complete order isomorphism from $\tilde{\cl S}$ onto $\text{\emph{OMAX}}_k(\cl S)$.
\end{thm}
\begin{proof} $(i)$ Assume $\phi: \text{OMAX}_k(\cl S) \to W$ is a k-positive map which is equivalent to $\phi$ being k-positive on the operator system $\cl S$ as a subspace of $B(\cl H)$ for some fixed Hilbert space $\cl H$, since $C_i^{k-max}(\cl S)=M_i(\cl S)^+$, $1\le i\le k$. Let $(a_{ij}) \in M_n(\text{OMAX}_k(\cl S))^+=C_n^{k-max}(\cl S)$:
\begin{itemize}
\item[(1)] If $(a_{ij})\in D_n^{k-max}(\cl S)$, then $(a_{ij})= ADA^*$ for some $A\in M_{n,mk}$, and some \\  $D=\text{diag}(D_1, D_2, \dots,D_m)$ where $D_i \in M_k(\cl S)^+$ for all $1\le i \le m,$ $ m\in \bb N$.\\ Then, we have \begin{eqnarray*}
\phi^{(n)}\big((a_{ij})\big) &=&\phi^{(n)}\big(ADA^*\big)= A\phi^{(mk)}(D)A^*\\ &=& A\text{ diag}\big(\phi^{(k)}(D_1),\,\phi^{(k)}(D_2),\dots, \phi^{(k)}(D_m)\big)A^* \in M_n(W)^+,\end{eqnarray*} since each $\phi^{(k)}(D_i)\in M_k(W)^+$ because $\phi:\cl S\to W$ k-positive.\\
\item[(2)] If  $(a_{ij})\in C_n^{k-max}(\cl S)$, then $re_n + (a_{ij})\in D_n^{k-max}(\cl S)$ for all $r>0$. It follows that $$\phi^{(n)}(re_n +(a_{ij}))= r(I_n\otimes \phi(e))+\phi^{(n)}\big((a_{ij})\big) \in M_n(W)^+\, \text{ for all } r>0.$$ Therefore by Lemma~\ref{rfactor}, we have $\phi^{(n)}\big((a_{ij})\big) \in M_n(W)^+$. \end{itemize} As a result, we conclude that $\phi:\text{OMAX}_k(\cl S)\to W$ is completely positive.\\\\
$(ii)$ Now, let $\tilde{\cl S} = (\cl S,\{C_n\}_{n=1}^\infty,e)$ be another operator system on $\cl S$ with $C_k=M_k(\cl S)^+$ such that for every operator system $W$, any k-positive map $\psi:\tilde{\cl S}\to W$ is completely positive. One can easily verify that $C_i=M_i(\cl S)$ for all $i=1,2,\dots,k$. Moreover, we know that $C_i^{k-max}(\cl S)=M_i(\cl S)^+$ for all $i=1,2,\dots,k$. This shows that both the identity map $\imath:\tilde{\cl S}\to \text{OMAX}_k(\cl S)$ and its inverse $\imath^{-1}:\text{OMAX}_k(\cl S)\to \tilde{\cl S}$ are k-positive maps. Then, by the assumption, $\imath$ is completely positive, and by part $(i)$, $\imath^{-1}$ is completely positive. Since $\imath$ is also unital, we have that $\tilde{\cl S}$ and OMAX$_k(\cl S)$ are completely order isomorphic.
\end{proof}

\begin{cor}\label{romaxk} Let $\cl S$ be an operator system with unit $e$, and $(W,\{P_n\}_{n=1}^\infty, e')$ be an (abstract) operator system. Then $\phi: \cl S \to W$ is k-positive if and only if $\phi: \text{\emph{OMAX}}_k(\cl S) \to W$ is completely positive.
\end{cor}

The following result gives an alternative way to describe the super k-maximal operator system structure $\cl C^{k-max}(\cl S)$:
\begin{prop} Let $\cl S$ be an operator system with unit $e$ and fix $k\in \bb N$. Then $(a_{ij})\in C^{k-max}_n(\cl S)$ if and only if $(\phi(a_{ij})) \in M_n(B(\cl H))^+$ for all unital k-positive maps $\phi:\cl S\to B(\cl H)$ and for all Hilbert spaces $\cl H$.
\begin{proof} Suppose $\phi:\cl S\to B(\cl H)$ is a k-positive map, where $\cl H$ is an arbitrary Hilbert space. Then $\phi:\text{OMAX}_k(V)\to B(\cl H)$ is completely positive by Theorem~\ref{omaxk}. For each $n\in \bb N$ set
\begin{eqnarray*}P_n^k(\cl S)&=& \big\{(a_{ij})\in M_n(\cl S): (\phi(a_{ij})) \in M_n(B(\cl H))^+ \text{ for all }\\ & & \phi:\cl S\to B(\cl H) \text{ unital k-positive }, \cl H \text{ Hilbert space }\big\}.\end{eqnarray*} It is clear that $C_n^{k-max}(\cl S)\subseteq P_n^k(\cl S)$ for all $n$. On the other hand, using Theorem ~\ref{c-e}, given the abstract operator system OMAX$_k(\cl S)$, there exists a Hilbert space $\cl H_0$, a concrete operator system $\cl S_0\subseteq B(\cl H_0)$ and a complete order isomorphism $\phi_0:\text{OMAX}_k(\cl S)\to \cl S_0\subseteq B(\cl H_0)$ with $\phi_0(e)=I_{\cl H_0}$. Then by Corollary~\ref{romaxk}, we have $\phi_0:\cl S\to B(\cl H_0)$ is unital k-positive. Let $(a_{ij})\in P_n^k(\cl S)$. Then $(\phi_0(a_{ij}))\in M_n(\cl S_0)^+\subseteq M_n(B(\cl H_0))^+$. It follows that $(a_{ij}) \in \phi_0^{-1}(M_n(\cl S_0)^+) \subseteq C_n^{k-max}(\cl S)$ since $\phi_0$ is a complete order isomorphism. Hence, $P_n^k(\cl S)\subseteq C_n^{k-max}(\cl S)$ for all $n$. As a result,
\begin{eqnarray*}C_n^{k-max}(\cl S)&=& \big\{(a_{ij})\in M_n(\cl S): (\phi(a_{ij})) \in M_n(B(\cl H))^+ \text{ for all }\\ & & \phi:\cl S\to B(\cl H) \text{ unital k-positive, } \cl H \text{ Hilbert space}\big\}.\end{eqnarray*}
\end{proof}
\end{prop}

\begin{prop} Let $\cl S$ be an operator system with unit $e$ and fix $k\in \bb N$. Then for $a\in \cl S$, we have that $$ \|a\|_{\text{\emph{OMAX}}_k(\cl S)}=\sup\{\|\varphi(a)\||\, \varphi: \cl S\to B(\cl H) \text{ unital k-positive}\},$$ where the supremum is taken over all Hilbert spaces and over all unital k-positive maps $\varphi$.
\end{prop}
\begin{proof} Suppose that $\varphi:\cl S \to B(\cl H)$ is a unital k-positive map. By Theorem ~\ref{omaxk}, $\varphi:\text{OMAX}_k(\cl S)\to B(\cl H)$ is completely positive and hence it is completely contractive. It follows that $\varphi$ is contractive and hence $$\|\varphi(a)\|\le \|a\|_{\text{OMAX}_k(\cl S)}, \text{ for all } a\in \cl S.$$ On the other hand, if $\varphi:\text{OMAX}_k(\cl S)\to B(\cl H)$ is a unital complete isometry, then $\varphi$ is completely positive and $\|a\|_{\text{OMAX}_k(\cl S)}=\|\varphi(a)\| , \text{ for all } a\in \cl S.$ Therefore, we conclude that $$ \|a\|_{\text{OMAX}_k(\cl S)}=\sup_{\cl H,\varphi}\{\|\varphi(a)\||\, \varphi: \cl S\to B(\cl H) \text{ unital k-positive }\}.$$
\end{proof}

\begin{prop}Let $\cl S$ be an operator system and let $k\in \bb N$. Then the identity map $\text{\emph{id}}$ from $\cl S$ to \emph{OMAX}$_k(\cl S)$  is completely bounded with $\|\text{\emph{id}}\|_{cb}=K$ if and only if for every operator system $\cl T$, every unital k-positive map $\phi:\cl S\to \cl T$ is completely bounded and the supremum of the completely bounded norms of all such maps is $K$.
\end{prop}
\begin{proof}Refer to \cite[Proposition 5.4]{PTT}.
\end{proof}

\begin{prop}Let $\cl S$ be an operator system with unit $e$. Then the identity map
from \emph{OMAX}$_k(\cl S)$ to \emph{MAX}$(\cl S_{k-min})$ is completely bounded if and only if for every Hilbert space $\cl H$, every bounded map $\phi: \cl S_{k-min}\to B(\cl H)$ decomposes as $$\phi=(\phi_1 -\phi_2)+i(\phi_3 -\phi_4),$$
where each $\phi_j :\cl S \to B(\cl H)$ is k-positive.
\end{prop}
\begin{proof} Assume that the decompositions of all such bounded maps holds, and suppose that  MAX$(\cl S_{k-min})\subseteq B(\cl H)$ completely isometrically for some Hilbert space $\cl H$. Let $\phi:\cl S_{k-min}\to \text{MAX}(\cl S_{k-min})$ be the identity map and let $\phi_j,\, j=1,2,3,4$ be a k-positive map on $\cl S$ such that $\phi=(\phi_1 -\phi_2)+i(\phi_3 -\phi_4)$. Since each $\phi_j$ is k-positive, then $\phi_j:\text{OMAX}_k(\cl S)\to B(\cl H)$ is completely positive, and hence completely bounded for $j=1,2,3,4$. Hence $\phi:\text{OMAX}_k(\cl S)\to B(\cl H)$ is completely bounded, too.\\ Conversely, if the identity map from OMAX$_k(\cl S)$ to MAX$(\cl S_{k-min})$ is completely bounded and $\phi: \cl S_{k-min}\to B(\cl H)$ is bounded, then $\phi:\text{MAX}(\cl S_{k-min})\to B(\cl H)$ is completely bounded and hence $\phi:\text{OMAX}_k(\cl S)\to B(\cl H)$ is completely bounded. Applying Wittstock's decomposition theorem, we have that $\phi=(\phi_1 -\phi_2)+i(\phi_3 -\phi_4),$ where each $\phi_j :\text{OMAX}_k(\cl S) \to B(\cl H)$ is completely positive, and therefore k-positive on $\cl S$.
\end{proof}

\section{The Matricial State Spaces of ${{\cl C}^{k-max}(\cl S)}$ and ${{\cl C}^{k-min}(\cl S)}$}\label{mat}

A matricial order on a $*$-vector space induces a natural matrix order on
its dual space. There is a correspondence between the various operator system structures that an AOU space can be endowed with and the corresponding matricial state spaces. Unfortunately, duals of AOU spaces are not in general AOU spaces, but they are normed $*$-vector spaces. 
As was shown in ~\cite{PT}, the order norm on the self-adjoint part $V_{sa}$ of an AOU space $(V,V^+,e)$ has many possible extensions to a norm on $V$, but all these norms are equivalent and hence the set of continuous linear functionals on $V$ with respect to any of these norms coincides with the same space which we shall denote by $V'$ and call {\bf the dual space of $V$}. For a functional $f\in V'$  we let $f^*\in V'$ be the functional given by $f^*(v) =\overline{f(v^*)}$; the mapping $f \to f^*$ turns $V'$ into a $*$-vector space.\\\\
\indent Given an AOU space $(V,V^+,e)$ and its dual $V'$, then let $M_{n,m}(V')$ denote the set of all $n\times m$ matrices with entries in $V'$, $n,m \in \bb N$. Then $M_{n,m}(V')$ together with natural addition and scalar multiplication is a complex vector space, which can be linearly identified as $M_{n,m}(V')\cong M_{n,m}\otimes V'\cong V'\otimes M_{n,m}$ by using the canonical matrix unit system $\{E_{i,j}\}_{i,j=1}^{n,m}$ of $M_{n,m}$. The direct sum and the matrix product operations that link these matrix linear spaces are defined in the same way as described in Section~\ref{prel}.\\\\
\indent Let $f:M_{n,m}(V)\to \bb C$ be a linear map on the complex vector space $M_{n,m}(V)$. We define $f_{ij}:V\to C$ by $f_{ij}(a)=f(E_{ij}\otimes a),\, a\in V$. Then for any $v=(v_{ij})\in M_{n,m}(V)$, we have $f(v)=\sum_{i,j}f_{ij}(v_{ij})$. We denote the vector space of such linear maps by $\cl L(M_{n,m}(V),\bb C)$.\\ Given $X=(x_{ki})\in M_{p,n}$ and $Y=(y_{jl})\in M_{m,q}$, we define $Xf:M_{p,m}(V)\to \bb C$ and $fY:M_{n,q}(V)\to \bb C$ by $$Xf=\bigg(\sum_{i=1}^nx_{ki}f_{ij}\bigg)_{k,j=1}^{p,m}\text{ and  } fY=\bigg(\sum_{j=1}^mf_{ij}y_{jl}\bigg)_{i,l=1}^{n,q},\text{ respectively}.$$
\begin{lemma} Let $(V,V^+,e)$ be an AOU space and $f:M_{n,m}(V)\to \bb C$ be a linear map, $n,m\in \bb N$. If $X\in M_{p,n}$ and $Y\in M_{m,q}$, $p,q\in \bb N$, then $Xf:M_{p,m}(V)\to \bb C$ and $fY:M_{n,q}(V)\to \bb C$ are linear and $$(Xf)(v)=f(X^tv)\text{   and   }(fY)(w)=f(wY^t),$$ where $v\in M_{p,m}(V)$ and $w\in M_{n,q}(V)$.
\begin{proof} Let $f=(f_{ij}):M_{n,m}(V)\to \bb C$ be a linear map and let $X=(x_{ki})\in M_{p,n}$ and $Y=(y_{jl})\in M_{m,q}$ be two arbitrary scalar matrices. It is trivial that both $Xf$ and $fY$ are linear functions on $M_{p,m}(V)$ and $M_{n,q}(V)$, respectively. Let $v=(v_{kj})\in M_{p,m}(V)$ and $w=(w_{il})\in M_{n,q}(V)$. Then we have
\begin{eqnarray*} (Xf)(v)&=&\sum_{k,j=1}^{p,m}(Xf)_{kj}(v_{kj})=\sum_{k,j=1}^{p,m}(\sum_{i=1}^nx_{ki}f_{ij})(v_{kj})\\ &=&\sum_{k,j=1}^{p,m}\sum_{i=1}^nx_{ki}f_{ij}(v_{kj})=\sum_{i,j=1}^{n,m}(\sum_{k=1}^pf_{ij}(x_{ki}v_{kj}))\\
&=&\sum_{i,j=1}^{n,m}f_{ij}(\sum_{k=1}^px_{ki}v_{kj})=\sum_{i,j=1}^{n,m}f_{ij}((X^tv)_{ij})=f(X^tv)\end{eqnarray*}
and
\begin{eqnarray*} (fY)(w)&=&\sum_{i,l=1}^{n,q}(fY)_{il}(w_{il})=\sum_{i,l=1}^{n,q}(\sum_{j=1}^m (f_{ij}y_{jl})(w_{il})\\ &=&\sum_{i,l=1}^{n,q}\sum_{j=1}^m f_{ij}(w_{il})y_{jl}=\sum_{i,j=1}^{n,m}(\sum_{l=1}^q f_{ij}(w_{il}y_{jl}))\\
&=&\sum_{i,j=1}^{n,m}f_{ij}(\sum_{l=1}^q w_{il}y_{jl})=\sum_{i,j=1}^{n,m}f_{ij}((wY^t)_{ij})=f(wY^t).\end{eqnarray*}
\end{proof}
\end{lemma}

 Let $(V,V^+, e)$ be an AOU space and let $f=(f_{ij}):M_{n,m}(V)\to \bb C$ be a linear map. There exists a linear map from the vector space of linear maps from $M_{n,m}(V)$, $\cl L(M_{n,m}(V),\bb C)$, into the vector space of linear maps from $V$ into $M_{n,m}$, denoted by $\cl L(V, M_{n,m})$, and vice versa. Hence, given $f\in \cl L(M_{n,m}(V),\bb C)$, we associate to $f$ a linear map $\phi_f=(f_{ij}):V\to M_{n,m}$ by the following formula:$$\phi_f(a)=(f_{ij}(a))\in M_{n,m}, \, a\in V.$$ On the other hand, given $\phi=(\phi_{ij})\in \cl L(V,M_{n,m})$, we associate to $\phi$ a linear map $f_{\phi}=(\phi_{ij}):M_{n,m}(V)\to \bb C$ by the following formula: $$f_\phi(v)=\sum_{i,j=1}^{n,m}\phi_{ij}(v_{ij})\in \bb C, \,v=(v_{ij})\in M_{n,m}(V).$$
Based on this correspondence between these vector spaces of linear maps, if $f=(f_{ij})\in \cl L(M_{n,m}(V),\bb C)$ with each $f_{ij}\in V'$, then $\phi_f=(f_{ij})$ can be regarded as sitting inside $M_{n,m}(V')$. Conversely, we identify $\phi=(\phi_{ij}) \in M_{n,m}(V')$  with the linear map $f_\phi: M_{n,m}(V) \to \bb C$ defined as above.\\

Let $\phi=(\phi_{ij}):V\to M_{n,m}$ be a linear map with $\phi_{ij}\in V'$. Given $A\in M_{n,p}$ and $B\in M_{m,q}$, $p,q\in \bb N$, then $A^*\phi B\in M_{p,q}(V')$ since both $M_{n,m}(V')$ and $M_{p,q}(V')$ are complex vector spaces and matrix product is a well-defined operation on them (see Section~\ref{prel}).\\ Write $A=\left[\begin{matrix}a_1 & a_2 &\cdots & a_p\end{matrix}\right]$ and $B=\left[\begin{matrix}b_1 & b_2 &\cdots & b_q\end{matrix}\right]$ where $a_k\in \bb C^n$ and $b_l\in \bb C^m$, $1\le k\le p,\,1\le l\le q$. Then we can write $A^*\phi B=(a^*_k\phi b_l)\in M_{p,q}(V')$, where each $a^*_k\phi b_l\in V'$. We identify $A^*\phi B\in M_{p,q}(V')$ with the linear map $F_{A^*\phi B}=(a_k^*\phi b_l):M_{p,q}(V)\to \bb C$ given by\\ $$F_{A^*\phi B}((v_{kl}))=\sum_{k,l=1}^{p,q}(A^*\phi B)_{kl}(v_{kl})=\sum_{k,l=1}^{p,q}(a^*_k\phi b_l)(v_{kl}),\, (v_{kl})\in M_{p,q}(V). $$ When $p=q=1$, we have $A\in \bb C^n$, $B\in \bb C^m$ and $A^*\phi B\in V'$. Moreover, $F_{A^*\phi B}:V\to \bb C$ is given by $A^*\phi B$ itself. One can straightforwardly show that $$F_{A^*\phi B}(a)=(A^*\phi B)(a)=A^*\phi(a) B, \text{ for all }a\in V.$$ The next lemma shows how to evaluate such maps when $p\ne 1,\,q\ne 1$. Before showing this result, we will discuss the matrix-vector correspondence and introduce a new notation which we will be using widely in the next results. \\

\noindent{\bf The Matrix -- Vector Correspondence}\label{mvc}\\

\noindent Let $X\in M_{n,m}$ be a matrix, $n,m\in \bb N$. Write $X$ in terms of its columns $X=\left[\begin{matrix}x_1 & x_2 & \cdots & x_m\end{matrix}\right]$ with $x_j \in \bb C^n$, $1\le j\le m$. We set $\text{vec}(X)=\left[\begin{matrix}x_1\\ x_2\\\vdots \\ x_m\end{matrix}\right]\in \bb C^m\otimes \bb C^n$ and call vec$(X)$ \\ {\bf the vectorization} of the matrix $X$. One can think of this process as a linear map $$\text{vec}:M_{n,m}\to \bb C^m\otimes \bb C^n \text{ given by vec}(E_{ij})=e_j\otimes e_i,$$ where $\{e_i\}_{i=1}^n\subseteq \bb C^n$ and $\{e_j\}_{j=1}^m\subseteq \bb C^m$ are the canonical orthonormal bases.

\begin{lemma} Let $(V,V^+,e)$ be an AOU space and $V'$ be its dual. If $\phi=(\phi_{ij})\in M_{n,m}(V')$, $A\in M_{n,p}$ and $B\in M_{m,q}$ are given, $n,m,p,q\in \bb N$, then the linear map $F_{A^*\phi B}:M_{p,q}(V)\to \bb C$ is given by $$F_{A^*\phi B}(v)=\text{\emph{vec}}(A)^*\phi^{(p,q)}(v)\text{\emph{vec}}(B), \text{ for all }v\in M_{p,q}(V).$$
\begin{proof} Let $\phi=(\phi_{ij})\in M_{n,m}(V)$, $A=\left[\begin{matrix} a_1 & a_2 &\cdots & a_p\end{matrix}\right]\in M_{n,p}$ and $B=\left[\begin{matrix} b_1 & b_2 &\cdots & b_q\end{matrix}\right]$ where $a_k\in \bb C^n$ and $b_l\in \bb C^m$, $1\le k\le p,\,1\le l\le q$, be given. Then $A^*\phi B\in M_{p,q}(V')$ and $F_{A^*\phi B}\in \cl L(M_{p,q}(V),\bb C)$. Let $v=(v_{kl})\in M_{p,q}(V)$, then we have
\begin{eqnarray*}F_{A^*\phi B}(v) &=&\sum_{k,l=1}^{p,q}(A^*\phi B)_{kl}(v_{k,l})=\sum_{k,l=1}^{p,q}(a^*_k\phi b_l)(v_{kl})\\ &=&\sum_{k,l=1}^{p,q} a^*_k\phi(v_{kl}) b_l=\left[\begin{matrix}a^*_1 & a^*_2 &\cdots & a^*_p\end{matrix}\right](\phi(v_{kl}))\left[\begin{matrix} b_1 \\ b_2 \\\vdots \\ b_q\end{matrix}\right]\\ &=&\text{vec}(A)^*\phi^{(p,q)}(v)\text{vec}(B).\end{eqnarray*}
\end{proof}
\end{lemma}

\begin{defn} Given an operator system structure $\{P_n\}_{n=1}^\infty$ on an AOU space $(V, V^+, e)$, then {\bf the dual of each cone $P_n$} is given by
$$P_n^d =\{ f:M_n(V) \to \bb C|\, f \text{ linear and } f(P_n)\subseteq \bb R^+ \}.$$ Given $f\in P_n^d$, we define $f_{ij}:V\to \bb C$ by $f_{ij}(v)=f(v\otimes E_{ij})$, where $E_{ij}$'s are the canonical matrix units for $M_n$.
\end{defn}

\indent Given an operator system structure $\{P_n\}_{n=1}^\infty$ on an AOU space $(V, V^+, e)$ and $f\in P_n^d$, then the functionals $f_{ij}$ belong to $V'$, as was shown in ~\cite{PTT}. Identifying each $f \in P^d_n$ with $(f_{ij}) \in M_n(V'),$  we shall regard $P_n^d$ as sitting inside $M_n(V')$.  \\

The dual cones of a given operator system structure $\{P_n\}_{n=1}^\infty$ on an AOU space $(V, V^+, e)$ form a matrix ordering on the dual normed space $V'$. Moreover, given a matrix ordering $\{Q_n\}_{n=1}^\infty$ on $V'$, one can construct an operator system structure on $V$ as the following result shows:

\begin{thm}\cite[Theorem 4.3]{PTT} Let $\{ P_n\}_{n=1}^\infty$ be an operator system structure on the AOU space $(V,V^+,e)$. Then $\{ P_n^d\}_{n=1}^\infty$ is a matrix ordering on the ordered $*$-vector space $V'$ with $P_1^d=(V^+)^d$. Conversely, if $\{Q_n\}_{n=1}^\infty$ is any matrix ordering on the $*$-vector space $V'$ with $Q_1=(V^+)^d$ and we set
$$ ^dQ_n=\{v\in M_n(V):\, f(v) \ge 0\, \text{ for all } \, f\in Q_n\}, $$
then $\{ ^dQ_n\}_{n=1}^\infty$ is an operator system structure on $(V,V^+,e)$.
\end{thm}

Note that the weak$^*$-topology on $V'$ endows $M_n(V')$ with a topology
which coincides with the weak$^*$-topology that comes from the identification
of $M_n(V')$ with the dual of $M_n(V)$. Thus, we shall refer to this topology,
unambiguously, as the weak$^*$-topology on $M_n(V')$.\\
The mappings $P_n \to P_n^d$ and $Q_n\to  {^dQ}_n$ establish a one-to-one inclusion-reversing correspondence between operator system structures $\{P_n\}_{n=1}^\infty$ on  $(V,V^+,e)$ and matrix orderings $\{Q_n\}_{n=1}^\infty$ on $V'$ with $Q_1=(V^+)^d$ for which each  $Q_n$ is weak$^*$-closed (see~\cite{PTT} for more details.)\\

Let $\cl S$ be an operator system with unit $e$. The natural operator system structure of $\cl S$ induces a natural matrix order on its dual space $\cl S'$, which makes $\cl S'$ an operator system too. The dual cones on $\cl S'$ can be described as  follows:
$$M_n(\cl S')^+=(M_n(\cl S)^+)^d=\{f:M_n(\cl S)\to \bb C\,|\, f \text{ positive linear functional }\},$$ for all $n\in \bb N$. Moreover, one can verify that $M_n(\cl S')^+\cong CP(\cl S, M_n)\text{ for all } n.$\\ Knowing the matricial state space of a given operator system $\cl S$, we would like to find the corresponding matricial state spaces of the k-minimal and the k-maximal operator systems.

\begin{defn} Let $\cl S$ be a given operator system with unit $e$. For a fixed $k\in \bb N$, set
\begin{multline*}Q_n^{k-min}(\cl S')=\big\{F_{X^*G X}:M_n(\cl S)\to\bb C \,\big| \, X\in M_{mk,n}, \\ G=\text{diag}(\phi_1,\dots,\phi_m) \text{ with }\phi_i\in CP(\cl S, M_k), m\in \bb N\bigg\},\end{multline*} and
$$Q_n^{k-max}(\cl S')=\big\{(f_{ij})\in M_n(\cl S'): \, \left(f_{ij}^{(k)}(a)\right)\in M_{nk}^+, \text{ for all } a \in M_k(\cl S)^+\,\big\}.$$
\end{defn}

\begin{prop}\label{dual} Let $\cl S$ be an operator system with unit $e$. Then $\{Q_n^{k-min}(\cl S')\}_{n=1}^\infty$ and $\{Q_n^{k-max}(\cl S')\}_{n=1}^\infty$ are matrix orderings on $\cl S'$ with $Q_1^{k-min}(\cl S')= (\cl S^+)^d$ and $Q_1^{k-max}(\cl S')=(\cl S^+)^d$.
\end{prop}

\begin{proof} One can straightforwardly check that both these families of cones are matrix orderings on $\cl S'$. Here, we will just show $Q_1^{k-min}(\cl S')=(\cl S^+)^d$ and $Q_1^{k-max}(\cl S')=(\cl S^+)^d$.\\\\
$(1)$ $Q_1^{k-min}(\cl S')= (\cl S^+)^d$:\\ Let $F_{X^* G X} \in Q_1^{k-min}(\cl S')\subseteq \cl S'$ with $X=\left[\begin{matrix}x_1\\\vdots\\x_m\end{matrix}\right] \in \bb C^{mk}$ where $x_i\in \bb C^k$, and\\ $G=\text{diag}(\phi_1,\dots,\phi_m)$ with $\phi_i \in CP(\cl S,M_k)$. Let $a\in \cl S^+$, then $$F_{X^*GX }(a)=(X^*GX)(a) = \bigg(\sum_{i=1}^m x_i^*\phi_ix_i\bigg)(a)=\sum_{i=1}^m x^*_i(\underbrace{\phi_i(a)}_{\ge 0})x_i \ge 0.$$ This implies that  $F_{X^*G X} \in (\cl S^+)^d$, i.e. $Q_1^{k-min}(\cl S')\subseteq (\cl S^+)^d$. \\Conversely, let $f\in (\cl S^+)^d$. Then the map $$\phi =I_k\otimes f=\left[\begin{matrix}{\begin{matrix}f& \\ & f\end{matrix}} & \text{\LARGE{0}} \\\text{\LARGE{0}} & {\begin{matrix}\ddots & \\ & f\end{matrix}}\end{matrix}\right]:\cl S\to M_k $$ is a well-defined completely positive linear map on $\cl S$.\\ Let $\alpha =\left[\begin{matrix}1\\ 0\\ \vdots\\ 0\end{matrix}\right]$, then $f=\alpha^* \phi \alpha \in Q_1^{k-min}(\cl S')$. Hence, $(\cl S^+)^d \subseteq Q_1^{k-min}(\cl S')$. \\As a result, $Q_1^{k-min}(\cl S')= (\cl S^+)^d$.\\\\
$(2)$ $Q_1^{k-max}(\cl S')= (\cl S^+)^d$:\\ Let $f\in Q_1^{k-max}(\cl S')\subseteq \cl S'$. Then, by the definition of $Q_n^{k-max}(\cl S')$, we have $f^{(k)}(a) \ge 0$ for all $a\in M_k(\cl S)^+$. Let $a= s\otimes E_{11} \in M_k(\cl S)^+$ with $s\in S^+$. Then
$$ f^{(k)}(a) = f(s)\otimes E_{11} \ge 0 \, \text{ implies } \, f(s) \ge 0 \text{ for all } s \in \cl S^+ .$$
Therefore, $f\in (\cl S^+)^d$. Hence,  $Q_1^{k-max}(\cl S')\subseteq (\cl S^+)^d$.\\ Conversely,
let $f \in (\cl S^+)^d$. Then $f: \cl S\to \bb C$ is completely positive. It follows $f^{(k)}(M_k(\cl S)^+)\subseteq M_k^+$, i.e. $f\in Q_1^{k-max}(\cl S')$ and $(\cl S^+)^d \subseteq Q_1^{k-max}(\cl S')$. As a result, $Q_1^{k-max}(\cl S')= (\cl S^+)^d$.
\end{proof}

\begin{thm}\label{duals} Let $\cl S$ be an operator system with unit $e$. Then $ ^dQ_n^{k-min}(\cl S')=C_n^{k-min}(\cl S)$ and $(C_n^{k-max}(\cl S))^d=Q_n^{k-max}(\cl S')$.
\end{thm}
\begin{proof} We will show that $ \{^dQ_n^{k-min}(\cl S')\}^\infty_{n=1} $ is the super k-minimal operator system structure on $\cl S$, and $\{Q_n^{k-max}(\cl S')\}^\infty_{n=1} $ is the dual of the super k-maximal operator system structure on $\cl S$.\\\\
$(1)$  $ ^dQ_n^{k-min}(\cl S')=C_n^{k-min}(\cl S)$:\\  Let $a=(a_{ij})\in C_n^{k-min}(\cl S)$ and  $F_{X^* G X} \in Q_n^{k-min}(\cl S')$, where  $X=\left[\begin{matrix} X_1 \\ X_2 \\ \vdots \\ X_m\end{matrix}\right] \in M_{mk,n}$ with each $X_i\in M_{k,n},\,1\le i\le m$, and $G=\text{diag}(\phi_1,\dots,\phi_m)$ with $\phi_i\in CP(\cl S,M_k)$. \\ One can easily check that $X^*GX=\sum_{i=1}^mX_i^*\phi_iX_i$ and $F_{X^*GX}=\sum_{i=1}^m F_{X_i^*\phi_iX_i}$. Then
\begin{eqnarray*}F_{X^*GX}(a)& =&\sum_{i=1}^m F_{X^*_i\phi_i X_i}(a)\\
&=&\sum_{i=1}^m  \text{vec}(X_i)^*\phi_i^{(n)}(a)\text{vec}(X_i)\ge 0,\end{eqnarray*}
since $\phi_i^{(n)}(a)\ge 0 $ for all $i$. This implies $a\in {^dQ}_n^{k-min}(\cl S')$ and $C_n^{k-min}(\cl S)\subseteq {^dQ}_n^{k-min}(\cl S')$. \\Conversely, let $a=(a_{ij})\in\, ^dQ_n^{k-min}(\cl S)$ and let $\phi\in S_k(\cl S)$. Let $\Lambda =\left[\begin{matrix}\lambda_1\\\vdots\\\lambda_n\end{matrix}\right] \in \bb C^{nk}$ with $\lambda_i \in \bb C^k,$ for all $ 1\le i \le n$. Then, we have
\begin{eqnarray*}\Lambda^* \phi^{(n)}(a)\Lambda &=&\sum_{i,j=1}^n\lambda_i^*\phi(a_{ij})\lambda_j=\sum_{i,j=1}^n(\lambda_i^*\phi\lambda_j)(a_{ij})\\ &=& F_{X^*\phi X}(a) \ge 0, \end{eqnarray*}
where $X=\left[\begin{matrix}\lambda_1 &\lambda_2 &\cdots \lambda_n\end{matrix}\right]\in M_{k,n}$, $X^*\phi X\in M_n(\cl S')$ and $F_{X^*\phi X} \in Q_n^{k-min}(\cl S')$. This implies $\phi^{(n)}(a)\ge 0$ for all unital k-positive maps $\phi $ on $\cl S$, i.e. $a=(a_{ij})\in C_n^{k-min}(\cl S)$ and $^dQ_n^{k-min}(\cl S')\subseteq C_n^{k-min}(\cl S)$. Hence, we conclude that  $ ^dQ_n^{k-min}(\cl S')=C_n^{k-min}(\cl S)$.\\\\
$(2)$ $(C_n^{k-max}(\cl S))^d=Q_n^{k-max}(\cl S')$:\\ Let $F=(f_{ij})\in Q_n^{k-max}(\cl S')$ and let $A^*DA \in D_n^{k-max}(\cl S)$. Write $A=\left[\begin{matrix}A_1 \\A_2 \\\vdots \\ A_m\end{matrix}\right]\in M_{mk,n}$ where each $A_l =\left[\begin{matrix}C_1^l& C_2^l& \cdots & C_n^l\end{matrix}\right]\in M_{k,n}$ with $C_i^l$ being the $i^{th}$ column of $A_l$ for all $1\le l \le m,\, 1\le i\le n$, and $D=\text{diag}(D_1, D_2, \cdots, D_m)$ with $D_l \in M_k(\cl S)^+$. Then, we have
\begin{eqnarray*}F(A^*DA)&=&\sum_{l=1}^m F(A_l^*D_lA_l)=\sum_{l=1}^m F([(C_i^l)^* D_l(C_j^l)])\\ &=&\sum_{l=1}^m \sum_{i,j=1}^n f_{ij}((C_i^l)^* D_l(C_j^l))= \sum_{l=1}^m \sum_{i,j=1}^n (C_i^l)^*\, f_{ij}^{(k)}(D_l)\,(C_j^l)\\ &=& \sum_{l=1}^m \text{vec}(A_l)^* \underbrace{\left[f_{ij}^{(k)}(D_l)\right]}_{\ge 0}\text{vec}(A_l) \ge 0. \end{eqnarray*}
 This shows that $F(D_n^{k-max}(\cl S))\subseteq \bb R^+$.  Now, let $a=(a_{ij}) \in C_n^{k-max}(\cl S)$ such that $re_n+ a \in D_n^{k-max}(\cl S)$ for all $r>0$. Then $$rF(e_n) + F(a) =F(re_n+a) \ge 0, \text{ for all } r>0.$$ Therefore, $F(a)\ge 0$ and $F(C_n^{k-max}(\cl S))\subseteq \bb R^+$. As a result, $F\in (C_n^{k-max}(\cl S))^d$ and $Q_n^{k-max}(\cl S')\subseteq (C_n^{k-max}(\cl S))^d$. \\ Conversely, let $F=(f_{ij})\in (C_n^{k-max}(\cl S))^d$, $a\in M_k(\cl S)^+$ and $\Lambda =\left[\begin{matrix}\lambda_1 \\\lambda_2 \\\vdots \\\lambda_n \end{matrix}\right] \in \bb C^{nk}$ with each $\lambda_i \in \bb C^k,\, 1\le i\le n$. Then, we have
\begin{eqnarray*} \Lambda^*\left[f_{ij}^{(k)}(a)\right]\Lambda
&=&\sum_{i,j=1}^n \lambda_i^* \,f_{ij}^{(k)}(a)\,\lambda_j= \sum_{i,j=1}^n f_{ij}(\lambda^*_i a\lambda_j)\\ &=&F([\lambda_i^* a\lambda_j])=F(X^*aX)\ge 0,\end{eqnarray*}
where $X=\left[\begin{matrix}\lambda_1&\lambda_2 &\cdots &\lambda_n\end{matrix}\right]\in M_{k,n}$ and $XaX^* \in C_n^{k-max}(\cl S)$. Therefore $\left[f_{ij}^{(k)}(a)\right] \ge 0$ for all $a\in M_k(\cl S)^+$, i.e. $F=(f_{ij})\in Q_n^{k-max}(V')$ and $(C_n^{k-max}(V))^d \subseteq Q_n^{k-max}(\cl S') $. It follows that $(C_n^{k-max}(\cl S))^d=Q_n^{k-max}(\cl S')$.
\end{proof}

\begin{remark} Although the positive cone $Q_n^{k-min}(\cl S')$ is not weak$^*$-closed, Theorem ~\ref{duals} shows that $(C_n^{k-min}(\cl S))^d$ is the weak$^*$-closure of $Q_n^{k-min}(\cl S')$.\\ Also, note that the cone $Q_n^{k-max}(\cl S')$ is weak$^*$-closed(easy to show) and $\{ ^dQ_n^{k-max}(\cl S')\}$ is an operator system structure on $\cl S$. This implies that $ (^dQ_n^{k-max}(\cl S'))^d =Q_n^{k-max}(\cl S')$.
\end{remark}

\section{The k-Partially Entanglement Breaking Maps}\label{kpeb}

In Quantum Information Theory, there is a great interest in quantum entanglement theory~\cite{HSR, HHHH} and the objects that support this theory like entangled states, separable states, and ``entanglement breaking'' maps. There is a well-known duality between the class of entanglement breaking maps and separable states defined on tensor composite systems. Based on this theory, a lot of work has been done to generalize the well-known class of entanglement breaking maps, and introducing the classes of ``partially entanglement breaking'' maps~\cite{CK, NJ}, which are related to ``partially separable states''. In this section, we will review these generalized concepts, and relate them to our construction of super minimal and super maximal operator system structures.\\

Let $M_n$ be the full algebra of $n\times n$ matrices, $n\in \bb N$. It is clear that $M_n$ is, in fact, an AOU space. Moreover, $M_n$ is an operator system arising from the identification of $M_n$ with $B(\bb C^n)$.  For some $k\in \bb N$, let OMIN$_k(M_n)$ be the super k-minimal operator system structure on $M_n$ and OMAX$_k(M_n)$ be the super k-maximal operator system structure on $M_n$. Then, we have $$M_m(\text{OMAX}_k(M_n))^+\subseteq M_m(M_n)^+\subseteq M_m(\text{OMIN}_k(M_n))^+, \text{ for all } m\in \bb N.$$ Note that OMIN$_k(M_n)$ is just the operator system $M_n\cong B(\bb C^n)$ for all $k\ge n$. The cone of positive elements of $M_n$ for any of these operator system structures coincides with the set of all positive definite matrices in $M_n$.\\

Let $s:M_n\otimes M_m \to \bb C$ be a (quantum) state defined on the composite system $M_n\otimes M_m,$ $n,m \in \bb N$. Then $s$ is called {\bf separable} if it is a convex combination of tensor states $$s=\sum_ir_i s_i\otimes t_i,$$ where $s_i:M_n\to \bb C$ and $t_i:M_m\to \bb C$ are states on the component systems, and $r_i \ge 0$ with $\sum_ir_i=1$. States that are not separable are said to be {\bf entangled}.\\

A state $s:M_n\otimes M_m \to \bb C$ can be represented by a positive semi-definite self-adjoint matrix operator of trace one, called a {\bf density matrix}. One commonly denotes density matrices with lowercase Greek letters such as $\rho, \xi, \sigma$. The density matrix of a quantum state $s:M_n\otimes M_m\to \bb C$ is in fact {\bf the Choi matrix} of the linear functional $s$, defined by $\rho_s=(s(E_{ij}\otimes E_{kl}))$, where $\{E_{ij}\}_{i,j=1}^n$ and $\{E_{kl}\}_{k,l=1}^m$ are the canonical matrix units for $M_n$ and $M_m$, respectively. This association of $s$ with its Choi matrix $\rho_s$ is an isomorphism known as Choi-Jamiolkowski isomorphism~\cite{MDC, AJ}. Being a positive-definite matrix, the density matrix $\rho_s$ can be written as a sum of rank one positive semi-definite matrices $\rho_s=\sum_{l=1}^p U_l U_l^*$, where $U_l \in \bb C^n\otimes \bb C^m.$ \\

We will classify quantum states according to their level of entanglement or separability. To measure the level of entanglement or separability in quantum states, we need to know the Schmidt number of the density matrix for the given state. \\

\noindent{\bf The Schmidt Number of a Density Matrix}
\begin{thm}\label{SDT}\cite[Schmidt Decomposition Theorem]{NC} Let $M_n$ and $M_m$ be Hilbert spaces of dimensions $n$ and $m$ respectively. For any vector $U$ in the tensor product $M_n\otimes M_m$, there exist orthonormal sets $\{u_1,u_2,\dots, u_k\} \subseteq M_n$ and $\{v_1,v_2,\dots, v_k\} \subseteq M_m$ for $k=\min(n,m)$, such that $$v=\sum_{i=1}^k\alpha_i u_i\otimes v_i, \text{ for some nonnegative real numbers }\alpha_i \ge 0. \qquad (\star)$$
\end{thm}

 The Schmidt Decomposition Theorem is a basic tool in quantum information theory. It is essentially the restatement of the Singular Value Decomposition in disguise. The standard proof of this theorem works by noticing that there is a linear isomorphism between $\bb C^n\otimes \bb C^m$ and $M_{n,m}$ given by associating a vector $u_e\otimes v_e\in \bb C^n\otimes \bb C^m$ with the matrix $u_ev_e^*\in M_{n,m}$ and extending linearly. We will denote the matrix associated to the vector $U$ by $A_u$. Applying the Singular Value Decomposition to $A_u$ gives the Schmidt Decomposition of $U$.\\
\indent In the Schmidt Decomposition \eqref{SDT} of $U$, the least number of terms required in the summation ($\star$) is known as the {\bf Schmidt rank} of $U$. One can realize that, the Schmidt rank of $U$ is equal to the number of non-zero singular values of the matrix $A_u$ associated to $U$, i.e. the rank of $A_u$. In a similar way, the nonnegative real constants $\alpha_e$'s are exactly the singular values of $A_u$, and they are often called the {\bf Schmidt coefficients}.\\

\indent Furthermore, since each $u_ev_e^* \in M_{n,m}$ has rank 1, we see that even if we remove the requirement that the sets above be orthonormal, it is impossible to write $U$ as the sum of fewer elementary tensors. To summarize, any vector $U\in \bb C^n\otimes \bb C^m$ can be written as $$U=\sum_{e=1}^k (u_e \otimes v_e),$$  for some sets of vectors $\{u_1,u_2,\dots,u_k\} \subseteq \bb C^n$ and $\{v_1,v_2,\dots,v_k\} \subseteq \bb C^m$ for $k\le \min(n,m)$. And any rank one positive semi-definite  matrix $UU^*$ can therefore be written as
$$UU^*=\sum_{e=1}^k (u_e\otimes v_e)\sum_{f=1}^k (u_f\otimes v_f)^*=\sum_{e,f=1}^k (u_eu_f^*\otimes v_ev_f^*).$$ In other words, given a vector $U$ of Schmidt rank at most k, we have $$UU^* \in \bigg\{\sum_{e,f=1}^k (u_eu_f^*\otimes v_ev_f^*):\{u_1,u_2,\dots,u_k\} \subseteq \bb C^n,\,\{v_1,v_2,\dots,v_k\} \subseteq \bb C^m\bigg\}. $$

\indent Let $s:M_n\otimes M_n\to\bb C$ be a quantum state and let $\rho_s$ be its density matrix. If the density matrix $\rho_s$ of the given state $s:M_n\otimes M_m\to \bb C$ is a finite sum of rank one positive semi-definite matrices $UU^*$ with $U \in \bb C^n\otimes \bb C^m$ of Schmidt rank at most k with $k\le \min(n,m)$, then the least such number k is called the {\bf Schmidt number}~\cite{TH} of $\rho_s$.\\

The Schmidt number of a density matrix tells us the ``level of entanglement or separability'' of the state. A state $s:M_n\otimes M_m\to \bb C$ is called {\bf maximally entangled} if the Schmidt number of its density matrix is $\min(n,m)$. Also, note that separable states are represented by density matrices of the form $\rho=\sum_j \sigma_j \otimes \tau_j$, where each $\sigma_j=\sum_eu^j_e(u^j_e)^*\ge 0,\,\tau_j=\sum_fv_f^j(v_f^j)^*\ge 0$. These are exactly the density matrices, whose Schmidt numbers are equal to $1$.\\

\indent A state $s:M_n\otimes M_m \to \bb C$ is called {\bf k-separable}~\cite{HHHH, TH} if the Schmidt number of its density matrix $\rho_s$ is at most k with $k\le \min(n,m)$. The quantum channels that carry any quantum states into k-separable states, are called {\bf k-partially entanglement breaking channels}.\\

\indent A nonzero positive linear functional $f:M_n\otimes M_m \to \bb C$ is called k-separable if and only if $\dfrac{f}{f(I_n\otimes I_m)}$ is a k-separable state. If $s:M_n\otimes M_m \to \bb C$ is a positive linear functional, then $s\circ \phi^{(n)}:M_n\otimes M_p \to \bb C$ is positive linear functional. If $s$ is a state and $\phi$ is unital, then $s\circ \phi^{(n)}$ is a state.

\begin{defn} A linear map $\phi:M_p \to M_m$ is called {\bf k-partially entanglement breaking} (k-PEB), if $s\circ \phi^{(n)} : M_n\otimes M_p \to \bb C$ is a k-separable state for every state $s:M_n\otimes M_m \to \bb C$, for all $n\in \bb N$.
\end{defn}

 In this section, we relate k-partially entanglement breaking maps to the k-minimal and the k-maximal operator system structures studied in the previous section. We begin with a characterization of k-separable states.

\begin{prop}\label{ksep} Let $f:M_n\otimes M_m \to \bb C$ be a  positive linear functional. Then $f$ is k-separable if and only if $f: M_n(\emph{OMIN}_k(M_m))\to \bb C$ is positive.\end{prop}
\begin{proof} Given a positive linear functional $f:M_n\otimes M_m\to \bb C$ with $f(I_n\otimes I_m)\ne 0$, then $\dfrac{f}{f(I_n\otimes I_m)}:M_n\otimes M_m\to \bb C$ becomes a state. Hence, we may assume $f$ is a k-separable state, $k\le \min(n,m)$. Assume that the density matrix of f is $$\rho_f =\sum_{e,f=1}^ku_eu_f^*\otimes v_ev_f^* ,$$ for some $\{u_1,u_2,\dots,u_k\} \subseteq \bb C^n$ and  $\{v_1,v_2,\dots,v_k\} \subseteq \bb C^m$. \\ Define $\phi_{ef}: M_m \to \bb C$ by $$\phi_{ef}(x)=\bar v_e^*(x)\bar v_f,\,\text{ for all } x\in M_m.$$ It is obvious that $\phi_{ef}$ is a well defined linear map on $M_m$. Note that the ``density matrix'' for each $\phi_{ef}$ is $$\rho_{ef} =\left[\phi_{ef}(E_{kl})\right]_{k,l=1}^m = v_ev_f^*.$$ Now, look at $\phi=\left[\phi_{ef}\right]:M_m\to M_k$ given by
\begin{eqnarray*}\phi(x)=\left[\phi_{ef}(x)\right]&=&\left[\bar v_e^*(x)\bar v_f\right]\\ &=&\left[\begin{matrix}\bar v_1^* \\ \vdots \\\bar v_k^*\end{matrix}\right]x\underbrace{\left[\begin{matrix} \bar v_1 & \cdots & \bar v_k\end{matrix}\right]}_{= A\in M_{m,k}}\\ &=& A^*xA.\end{eqnarray*}
Then, one can easily verify that $\phi$ is a completely positive map on $M_m$. Hence, we can write each function $f$ as $$f=\sum_{e,f=1}^k u_eu_f^*\otimes \phi_{ef}=\left[\begin{matrix}u_1 & \cdots & u_k\end{matrix}\right] \phi \left[\begin{matrix}u_1^* \\ \vdots \\ u_k^*\end{matrix}\right].$$ This shows that $f\in {Q}_n^{k-min}(M_m)$, i.e. $f\in (C_n^{k-min}(M_m))^d$.\\So, $f$ is positive on $M_n(\text{OMIN}_k(M_m))$.\\Conversely, assume  $f: M_n(\text{OMIN}_k(M_m))\to \bb C$ is positive , i.e. $f\in (C_n^{k-min}(M_m))^d =\overline{Q_n^{k-min}(M_m)}^{w^*}$. Without loss of generality, let $f=\Lambda \phi\Lambda^* \in {^sQ}_n^{k-min}(M_m)$, where $\Lambda=\left[\begin{matrix}u_1 & u_2 & \cdots u_k\end{matrix}\right]\in M_{n,k}$ and $\phi=[\phi_{ef}]:M_m\to M_k$ is completely positive. Then $f=\sum_{e,f=1}^ku_eu_f^* \otimes \phi_{ef}$. Since $\phi$ is completely positive, then $\phi(x)=\sum_{i=1}^l A_i^*xA_i$, for some Kraus operators $\{A_i\}\subseteq M_{m,k}$. Writing each $A_i =\left[\begin{matrix}\bar v_1^i & \bar v_2^i & \cdots & \bar v_k^i\end{matrix}\right]$, where each $\bar v_e^i \in \bb C^m$, then one can see that $\phi_{ef}(x) = \sum_{i=1}^l(\bar v_e^i)^*x(\bar v_f^i)$, and its density matrix $\rho_{\phi_{ef}} = \sum_{i=1}^l v_e^i (v_f^i)^*$. Hence, the density matrix for the function $f$ will be $$\rho_f=\sum_{i=1}^l\sum_{e,f=1}^ku_eu_f^*\otimes v_e^i (v_f^i)^*. $$ This shows that f is a k-separable map.\\
In general, any positive linear functional $f\in (C_n^{k-min}(M_m))^d$ (which becomes a state by dividing by its norm) is a weak$^*$-limit of k-separable states. Such a limit exists, because k-separable states are the convex hull of a compact set, which is a compact set by Caratheodory's theorem.
\end{proof}

We now turn our attention to a duality result. Recall that the dual of a
matrix ordered space is again a matrix ordered space. Let $\delta_{i,j}: M_n\to \bb C$ be the linear functional satisfying
$$\delta_{i,j}(E_{kl}) =\begin{cases} 1 & \text{when }(i,j)=(k,l)\\
0 & \text{when }(i,j)\ne (k,l)\end{cases}$$
and let $\gamma_n: M_n\to M_n'$ be the linear isomorphism defined by $\gamma_n(E_{i,j})=\delta_{i,j}$. The next result is certainly in some sense known, but the formal statement will be useful for us in the sequel.

\begin{thm}\cite[Theorem 6.2]{PTT} The map $\gamma_n : M_n\to M_n'$ is a complete order isomorphism of matrix ordered spaces. Consequently, $(M_n', (M_n')^+, tr)$ is an AOU space that is order isomorphic to $(M_n,M_n^+, I_n)$, where $I_n$ denotes the identity matrix.
\end{thm}

\begin{prop}\label{duality} The complete order isomorphism $\gamma_n : M_n\to M_n'$ gives rise to the identifications $\emph{OMIN}_k(M_n)'=\emph{OMAX}_k(M_n')=\emph{OMAX}_k(M_n)$ and $\emph{OMAX}_k(M_n)'=\emph{OMIN}_k(M_n')=\emph{OMIN}_k(M_n)$.
\end{prop}
\begin{proof} Let $\cl S=M_n'$, then one can observe that ${Q}_m^{k-min}(M_n)={D}_m^{k-max}(\cl S)$ by definitions of each cone. The unit ball of ${D}_m^{k-max}(\cl S)$ is compact, therefore ${D}_m^{k-max}(\cl S)$ is closed by the Krein-Shmulian Theorem. Hence, ${D}_m^{k-max}(\cl S)={C}_m^{k-max}(\cl S)$. Thus, we have that ${Q}_m^{k-min}(M_n)={C}_m^{k-max}(\cl S)$, i.e.\\ $M_m(\text{OMIN}_k(M_n)')^+=M_m(\text{OMAX}_k(M_n'))^+$, and so the identity map on $M_n'$ yields a complete order isometry between the matrix ordered space OMIN$_k(M_n)'$ and the operator system OMAX$_k(M_n')$. Finally, the complete order isomorphism $\gamma_n$ allows for the identification, OMAX$_k(M_n')=\text{OMAX}_k(M_n)$. The proof of the rest of the statement is similar.
\end{proof}

\begin{thm}\label{kpebmin} Let $\phi:M_p\to M_m$ be a linear map. Then $\phi$ is a k-partially entanglement breaking map if and only if $\phi:\emph{OMIN}_k(M_p)\to M_m$ is completely positive.
\end{thm}
\begin{proof} Assume $\phi:\text{OMIN}_k(M_p)\to M_m$ is completely positive. Then $\phi':M_m'\to \text{OMIN}_k(M_p)'$ is completely positive too. If $f=(f_{ij})\in M_n(M_m')^+$ is any state on $M_n\otimes M_m$, then $\left[\phi'(f_{ij})\right] \in M_n(\text{OMIN}_k(M_p)')^+$. By Proposition ~\ref{ksep}, these are exactly the k-separable states on $M_n\otimes M_p$, i.e. $$f\circ \phi^{(n)}=\left[\phi'(f_{ij})\right]=\left[f_{ij}\circ \phi\right]:M_n\otimes M_p \to \bb C$$ is k-separable, for all states $f:M_n\otimes M_m\to \bb C$. This implies $\phi$ is a k-PEB map. Conversely, assume $\phi$ is k-PEB. Then, for any $f=(f_{ij})\in M_n(M_m')^+$, we have $f\circ \phi^{(n)}$ is k-separable, i.e. $f\circ \phi^{(n)}=\left[\phi'(f_{ij})\right]\in M_n(\text{OMIN}_k(M_p)')^+$, which implies that $\phi': M_m'\to \text{OMIN}_k(M_p)'$ is completely positive. As a result, we have $\phi:\text{OMIN}_k(M_p)\to M_m$ is completely positive.
\end{proof}

\vspace{0.5cm}

\noindent{\bf Note:} Let $U=\left[\begin{matrix}u_1\\ u_2\\\vdots\\ u_k\end{matrix}\right]=\sum_{j=1}^ke_j\otimes u_j \in \bb C^k\otimes \bb C^m$, where $u_j\in \bb C^m$, $1\le j\le k$. Then $U$ can be viewed as the $m\times k$ matrix $M_u=\left[\begin{matrix}u_1 &u_2 &\cdots & u_k\end{matrix}\right]\in M_{m,k}$. If $\lambda \in \bb C^k$, then we have $(\lambda^*\otimes I_m)(UU^*)(\lambda\otimes I_m)=M_u(\lambda \lambda^*)M_u^*.$

\begin{prop}\label{kpebmax}  Let $\phi:M_p\to M_m$ be a linear map. Then $\phi:M_p\to \emph{OMAX}_k(M_m)$ is completely positive if and only if there exist completely positive maps $\psi_l:M_p\to M_k$ and matrices $M_l \in M_{m,k}$, $l=1,\dots, q$ such that $\phi(X)=\sum_{l=1}^q M_l \psi_l(X) M_l^*$.
\end{prop}
\begin{proof} We have that  $\phi:M_p\to \text{OMAX}_k(M_m)$ is completely positive if and only if $(\phi(E_{ij}))\in M_p(\text{OMAX}_k(M_m))^+={C}_p^{k-max}(M_m)={D}_p^{k-max}(M_m)$, since the set ${D}_p^{k-max}(M_m)$ is closed. Thus, there exists an integer $q$, $A_1,\dots, A_q \in M_{k,p}$, positive matrices $D_1, \dots, D_q \in M_k(M_m)^+$, such that $$(\phi(E_{ij}))=\sum_{l=1}^q(A_l^* \otimes I_m)D_l(A_l\otimes I_m).$$  Write $A_l=\left[\begin{matrix}\lambda_{1,l} &\lambda_{2,l} &\cdots &\lambda_{p,l}\end{matrix}\right]$, where $\lambda_{i,l} \in \bb C^k$ for all $i=1,\dots,p$. Then, we have $\phi(E_{ij})=\sum_{l=1}^q (\lambda_{i,l}^*\otimes I_m)D_l(\lambda_{j,l}\otimes I_m)$. Since $D_l\in M_k(M_m)^+$, then $D_l=\sum_{r=1}^t U_{r,l} U_{r,l}^*$, where $U_{r,l}\in \bb C^k\otimes \bb C^m$ for all $1\le r\le t$. Without loss of generalization, assume $D_l=U_lU_l^*$, where $U_l=\left[\begin{matrix}u_{1,l}\\ u_{2,l}\\\vdots\\ u_{k,l}\end{matrix}\right]$, each $u_{e,l}\in \bb C^m$, for all $1\le e\le k$. This implies
$$\phi(E_{ij})=\sum_{l=1}^q(\lambda_{i,l}^*\otimes I_m)D_l(\lambda_{j,l}\otimes I_m)=\sum_{l=1}^q M_l\left[ (\bar \lambda_{i,l}) (\bar\lambda_{j,l})^*\right]M_l^*,$$
where $M_l=\left[\begin{matrix}u_{1,l} & u_{2,l} &\cdots & u_{k,l}\end{matrix}\right]\in M_{m,k}$ is the corresponding matrix for $U_l$. If we define completely positive maps $\psi_l:M_p \to M_k$ by $$\psi_l(X)=\sum_{i,j=1}^p (\bar \lambda_{i,l}) x_{ij} (\bar \lambda_{j,l})^*=\bar A_l X\bar A_l^*,$$ then we have that $\phi(E_{ij})=\sum_{l=1}^q M_l \psi_l(E_{ij})M_l^*$, for all $1\le i,j\le p$, and hence $\phi(X)=\sum_{l=1}^q M_l \psi_l(X) M_l^*$ for every $X\in M_p$.\\
Conversely, given any completely positive map $\psi:M_p\to M_k$, then $\psi$ can be written as $\psi(X)=(\bar A) X (\bar A)^* = \sum_{i,j=1}^p (\bar\lambda_i)x_{ij}(\bar\lambda_j)^*$, where $A=\left[\begin{matrix}\lambda_1 &\lambda_2 &\cdots &\lambda_p\end{matrix}\right]\in M_{k,p}$ with $\lambda_i\in \bb C^k$. Thus, if $\phi(X)=\sum_{l=1}^q M_l \psi_l(X)M_l^*$, where $M_l=\left[\begin{matrix}u_{1,l} & u_{2,l} &\cdots & u_{k,l}\end{matrix}\right]\in M_{m,k}$ with $u_{e,l}\in \bb C^m$ for all $1\le e\le k$, and $\psi_l:M_p\to M_k$ completely positive, then by increasing the number of terms in the sum we may assume that each $\psi_l$ has the form $\psi_l(X)=(\bar A_l)X(\bar A_l)^*=\sum_{i,j=1}^p (\bar \lambda_{i,l}) x_{ij} (\bar \lambda_{j,l})^*$, and hence $$\phi(E_{ij})=\sum_{l=1}^q M_l\left[ (\bar \lambda_{i,l}) (\bar\lambda_{j,l})^*\right]M_l^* =\sum_{l=1}^q(\lambda_{i,l}^*\otimes I_m)D_l(\lambda_{j,l}\otimes I_m),$$\\ where $D_l=\left[\begin{matrix}u_{1,l}\\ u_{2,l}\\\vdots\\ u_{k,l}\end{matrix}\right]\left[\begin{matrix}u^*_{1,l} & u^*_{2,l}& \cdots &  u^*_{k,l}\end{matrix}\right]=U_lU_l^* \in M_k(M_m)^+$. \\Thus $(\phi(E_{ij}))=\sum_{l=1}^q (A_l^*\otimes I_m)D_l(A_l\otimes I_m) \in {D}_p^{k-max}(M_m)$, and it follows that $\phi:M_p\to \text{OMAX}_k(M_m)$ is completely positive.

\end{proof}

\begin{cor}\label{maxkpeb} If $\phi:M_p\to \emph{OMAX}_k(M_m)$ is completely positive, then $\phi$ is a k-partially entanglement breaking map.
\end{cor}
\begin{proof} By Proposition~\ref{kpebmax}, there exist completely positive maps $\psi_l:M_p\to M_k$ and matrices $M_l \in M_{m,k},\, 1\le l\le q$, such that $\phi(X)=\sum_{l=1}^q M_l\psi_l(X)M_l^*$. Given any $n\in \bb N$ and any positive linear functional $f:M_n\otimes M_m\to \bb C$, we have $f\circ \phi^{(n)} :M_n\otimes M_p\to \bb C$ is k-separable if and only if $f\circ \phi^{(n)} :M_n(\text{OMIN}_k(M_p))\to \bb C$ is a positive linear functional by Proposition~\ref{ksep}. Let $(X_{ij})\in {C}_n^{k-min}(M_p)$, then we have
\begin{eqnarray*}\phi^{(n)}((X_{ij}))=\left(\phi(X_{ij})\right)&=&\sum_{l=1}^q \left(M_l\psi_l(X_{ij})M_l^*\right)\\ &=& \sum_{l=1}^q (I_n\otimes M_l)\psi^{(n)}((X_{ij}))(I_n\otimes M_l^*) \ge 0,\end{eqnarray*}
since $\psi^{(n)}((X_{ij}))\ge $ for all $(X_{ij})\in {C}_n^{k-min}(M_p)$. Thus, $(f\circ\phi^{(n)})((X_{ij}))=f((\phi(X_{ij}))) \ge 0$ since $f$ is a positive linear functional on $M_n\otimes M_m$ and $\left(\phi(X_{ij})\right)\in M_n(M_m)^+$. As a result, $f\circ \phi^{(n)}$ is k-separable, which implies that $\phi$ is k-PEB.
\end{proof}

\begin{thm} Let $\phi:M_p\to M_m$ be a linear map, and $k \le \min(p,m)$. Then the following are equivalent:
\begin{itemize}
\item[(i)] $\phi:\emph{OMIN}_k(M_p)\to M_m$ is completely positive.
\item[(ii)]$\phi$ is k-partially entanglement breaking.
\item[(iii)] $\phi: M_p\to \emph{OMAX}_k(M_m)$ is completely positive.
\item[(iv)] There exist completely positive maps $\psi_l:M_p\to M_k$ and $M_l\in M_{m,k}$, for $1\le l\le q$ such that $\phi(X)=\sum_{l=1}^q M_l\psi_l(X)M_l^*$.
\item[(v)] There exist matrices $A_l \in M_{p,m},\, 1\le l\le r$ of rank at most $k$, such that $\phi(X)=\sum_{l=1}^s A_l^* XA_l$.
\item[(vi)] $\phi:\emph{OMIN}_k(M_p)\to \emph{OMAX}_k(M_m)$ is completely positive.
\end{itemize}
\end{thm}
\begin{proof}The equivalence of $(i)$ and $(ii)$ is stated in Theorem~\ref{kpebmin}, while the equivalence of $(iii)$ and $(iv)$ is stated in Proposition~\ref{kpebmax}. By Corollary~\ref{maxkpeb}, $(iii)$ implies $(ii)$. Note that $\phi:M_p\to \text{OMAX}_k(M_m)$ is completely positive if and only if $\phi': \text{OMAX}_k(M_m)'\to M_p'$ is completely positive. Using the identifications of Proposition~\ref{duality}, we have $\phi':\text{OMIN}_k(M_m')\to M_p'$  is completely positive if and only if $\phi^\flat =\gamma_p^{-1}\circ \phi'\circ \gamma_m :\text{OMIN}_k(M_m) \to M_p$ is completely positive, i.e. $\phi^\flat:M_m\to M_p$ is k-PEB. Hence, if $\phi=(\phi^\flat)^\flat$ is k-PEB, then $\phi^\flat$ is k-PEB, which is equivalent to $\phi:M_p\to \text{OMAX}_k(M_m)$ is completely positive. So $(ii)$ implies $(iii)$. Now we have the equivalence $(i)-(iv)$.\\To show that $(iv)$ implies $(v)$, we may assume that each completely positive map $\psi_l:M_p\to M_k$ can be written as $\psi_l(X)=\sum_{j=1}^rB^*_{j,l}XB_{j,l}$, for some $B_{j,l} \in M_{p,k}$. Then,
$$\phi(X)=\sum_{l=1}^q \sum_{j=1}^r M_l B_{j,l}^*X B_{j.l}M_l^*= \sum_{l=1}^s A_l^*X A_l,$$
where each $A_l= B_{j,l}M_l^* \in M_{p,m}$ has rank at most k for all $1\le l\le s$, since $\text{rank}(A_l) \le \min(\text{rank}(B_{j,l}),\text{rank}(M_l))=k$.\\ To see that $(v)$ implies $(iv)$, each $A_l\in M_{p,m}$ of rank at most k, can be factorized as $A_l=B_lM_l$, where $M_l \in M_{k,m}$ is the reduced matrix of $A_l$ containing only the k rows that span $A_l$, and $B_l\in M_{p,k}$ is the coefficient matrix of $A_l$. Set $\psi_l(X)=B_l^*XB_l$ which is completely positive, then $\phi(X)=\sum_{l=1}^s M_l^*\psi_l(X)M_l$.\\Finally, clearly $(vi)$ implies $(i)$. One can easily check that $(iv)$ implies $(vi)$.
\end{proof}

\vspace{0.1in}

\noindent{\bf Acknowledgements.} I owe my deepest gratitude to my graduate advisor Dr. Vern Paulsen for his supervision and support during my graduate degree at the University of Houston. His unique perspectives and visualizations for the subjects were the most significant contribution to this paper.  


\end{document}